\input ./style/arxiv-general.cfg
\documentclass[aos,MSNbibl,nameyear,dvips]{arximspdf}
\makeatletter
   \@ifpackageloaded{graphicx}{}{\usepackage{graphicx}}
\makeatother

\doi{10.1214/14-AOS1251}% Updated by VTEXPTS2LaTeX.exe, 26.02.2015
\volume{43}
\issue{3}
\pubyear{2015}
\firstpage{937}
\lastpage{961}
\docsubty{FLA}

\makeatletter
\newcommand{\eqref}[1]{(\ref{#1})}
\def\ci{\perp\!\!\!\perp}

\newtheorem{theorem}{Theorem}
\newtheorem{Theoremm}{Theorem} %1a

\newtheorem{lemma}[theorem]{Lemma}

\newtheorem{proposition}[theorem]{Proposition}
\newcommand{\E}{\mathrm{E}}
\newcommand{\X}{X}
\newcommand{\Y}{Y}
\newcommand{\thet}{\theta}
\newcommand{\opt}{\mathrm{opt}}
\newcommand{\phat}{\hat{p}}
\newcommand{\thetab}{\theta}
\newcommand{\that}{\hat{\theta}}
\newcommand{\dhat}{\hat{d}}
\newcommand{\x}{x}
\newcommand{\y}{y}

\newcommand{\smallo}{o}
\newcommand{\RR}{\mathbb{R}}
\newcommand{\hf}{ {\frac{1}{2}}}
\newcommand{\mR}{\mathbb{R}}
\newcommand{\me}{\mathfrak{m} (\eta)}

\makeatother

\begin{document}
\begin{frontmatter}

\title{Exact minimax estimation of the predictive density in
sparse Gaussian models\thanksref{T1}}
\runtitle{Sparse predictive density estimation}
\thankstext{T1}{Supported in part by NSF Grant DMS-09-06812
and NIH Grant R01 EB001988.}

\begin{aug}
% Corresponding author: Iain Johnstone - imj@stanford.edu% Updated by
%VTEXPTS2LaTeX.exe, 26.02.2015 15:38
\author[A]{\fnms{Gourab} \snm{Mukherjee}\ead[label=e1]{gourab@usc.edu}}
\and
\author[B]{\fnms{Iain M.} \snm{Johnstone}\corref{}\ead[label=e2]{imj@stanford.edu}}
\runauthor{G. Mukherjee and I.~M. Johnstone}
\affiliation{University of Southern California and Stanford University}

\address[A]{Department of Data Sciences and Operations\\
Marshall School of Business \\
University of Southern California\\
Los Angeles, California 90089-0809\\
USA\\
\printead{e1}}
\address[B]{Department of Statistics\\
Sequoia Hall, 390 Serra Mall\\
Stanford University\\
Stanford, California 94305-4065\\
USA\\
\printead{e2}}
\end{aug}

% HISTORY:
%
\received{\smonth{1} \syear{2014}}% Updated by VTEXPTS2LaTeX.exe,
%26.02.2015 15:38
%
\revised{\smonth{6} \syear{2014}}% Updated by VTEXPTS2LaTeX.exe,
%26.02.2015 15:38

% ABSTRACT
%
\begin{abstract}
We consider estimating the predictive density under Kullback--Leibler
loss in an $\ell_0$ sparse Gaussian sequence model. Explicit
expressions of the first order minimax risk along with its exact
constant, asymptotically least favorable priors and optimal
predictive density estimates are derived. Compared to the sparse
recovery results involving point estimation of the normal mean, new
decision theoretic phenomena are seen. Suboptimal performance
of the class of plug-in density estimates reflects the predictive
nature of the problem and optimal strategies need diversification of
the future risk. We find that minimax optimal strategies lie outside
the Gaussian family but can be constructed with threshold predictive
density estimates. Novel minimax techniques involving simultaneous
calibration of the sparsity adjustment and the risk diversification
mechanisms are used to design optimal predictive density estimates.
\end{abstract}

% KEYWORDS
\end{frontmatter}

%s1 #&#
\section{Introduction}\label{sec-intro}

Statistical prediction analysis aims to use past data
to choose a probability distribution that will be good in predicting
the behavior of future samples.
This well-established subject [\citet{Aitchison-book,Geisser-book}]
finds application in game theory, econometrics, information theory,
machine learning, mathematical finance, etc.

In this paper we study predictive density estimation in a
high-dimensional setting and, in particular, explore the consequences
of sparsity assumptions on the unknown parameters.

%s1.1 #&#
\subsection{Main results}
\label{sec:main-results}

We begin by describing some of our main results: fuller references,
background and interpretation follow in Section~\ref{sec:backgr-prev-work}.

We work in the simplest Gaussian model for high-dimensional
prediction:
%
%e1 #&#
\begin{equation}
\label{eq:modelM2} \X\sim N_n(\thet, v_x I), \qquad\Y\sim
N_n(\thet, v_y I),\qquad \X\ci\Y| \thet.
\end{equation}
On the basis of the ``past'' observation vector $\X$, we seek to predict
the distribution of a future observation $\Y$.
%independent of $\X$, with joint density $p(\y| \thet)$.
% Assume that the `past' observation vector $\X\sim N_n(\thet, v_x
% I)$ with joint density $p(\x| \thet)$.
% On the basis of $\X$ we seek to predict the distribution of a future
% observation $\Y\sim N_n(\thet, v_y I)$,
% independent of $\X$, with joint density $p(\y| \thet)$.
The past and future observations are independent, but are
linked by the common mean parameter
$\thet$, assumed to be unknown.
Note, however, that the variances, assumed here to be known, may
differ.
We write $p({x}| \thet, v_x)$ and $p({y}|
\thet, v_y)$ for the probability
densities of $\X$ and $\Y$, respectively.

We seek estimators $\hat p({y}|{x})$ of the
future observation density
$p({y}|\thet, v_y)$, and to compare their performance
under sparsity
assumptions on $\thet$.
We recall two natural ways of generating large classes of estimators.
Perhaps simplest are the ``plug-in'' or estimative densities: given a
point estimate $\hat\thet(\X)$, simply set
$\hat p({y}|{x}) = p({y}|\hat\thet)$.
We often use the abbreviation $p[\that]$.
% We represent plug-in predictive density estimate based on $\hat
% \thet(\X)$ location estimate by $p[\hat\thet]$.
%$p[\hat\thet](\y|\x) = p(\y|\hat\thet)$.
Second, given any prior measure $\pi( d \thet)$, proper or improper,
such that the posterior $\pi(d \thet| {x})$ is well
defined, the Bayes
predictive density is
%
%e2 #&#
\begin{equation}
\label{eq:bpd} \hat p_\pi({y}|{x}) = \int p(
{y}|\thet, v_y) \pi(d \thet|{x}).
\end{equation}
The important case of a uniform prior measure $\pi(d \thet) = d
\thet$ leads to predictive density
$\hat p_U({y}|{x})$, easily seen to correspond to
$N_n({x}, (v_x + v_y)I)$.

We will examine similarities and differences between high-dimensional
prediction and high-dimensional estimation.
In particular,
%the uniform prior predictive density
$\hat p_U({y}|{x})$ plays in prediction
the role of the maximum likelihood estimator
$\hat\thet_{\mathrm{MLE}} ({x}) = {x}$ in the
multinormal mean estimation
setting.
In contrast to the corresponding plug-in estimate $p[\that_{\mathrm{MLE}}]$,
the density $\hat p_U$ incorporates the variability of the location
estimate which leads to a flattening of the estimator:
$v_x + v_y > v_y$.
% Unlike, its corresponding plugin density estimate $\hat p[x]$, $\hat
%p_U$ incorporates the variability of the location estimate which leads
%to a flattening of the density estimate.

To evaluate the performance of a predictive density estimator $\hat
p({y}|{x})$, we use the familiar
Kullback--Leibler ``distance'' as loss
function:
\[
L\bigl(\thet, \hat p(\cdot|{x})\bigr) = \int p({y}|\thet,
v_y) \log\frac{p(
{y}|\thet, v_y)}{\hat p({y}|{x})} \,d{y}.
\]
The corresponding K--L risk function follows by averaging over the
distribution of the past observation:
\[
\rho(\thet, \hat p) = \int L\bigl(\thet, \hat p(\cdot|{x}) p(
{x}|\thet, v_x)\bigr) \,d {x}.
\]

Given a prior measure $\pi( d\thet)$, the average or integrated risk
is
%
%e3 #&#
\begin{equation}
\label{eq:integ-risk} B(\pi, \hat p) = \int\rho(\thet, \hat p) \pi(d \thet).
\end{equation}
The Bayes predictive density (\ref{eq:bpd}) can be shown to
minimize both the posterior expected loss $\int L( \thet, \hat p(\cdot|
{x})) \pi(d \thet|{x})$ and the integrated
risk $B(\pi, \hat p)$ in the
class of all density estimates.
This is a general fact in statistical decision theory [\citet{Brown-notes}],
the resulting minimum the Bayes K--L risk:
%
%e4 #&#
\begin{equation}
\label{def-bayes-risk} B(\pi) = \inf_{\hat p} B(\pi, \hat p).
\end{equation}

Our main focus is on how to optimize the predictive risk $\rho( \thet,
\hat p)$ in a high-dimensional setting under an $\ell_0$-sparsity
condition on the parameter space.
Thus, let $\| \thet\|_0 = \# \{ i\dvtx\theta_i \neq0 \}$ and
%
%e5 #&#
\begin{equation}
\label{eq:theta-n-s} \Theta_n[s] = \bigl\{ \thet\in\mathbb{R}^n
\dvtx\| \thet\|_0 \leq s \bigr\}.
\end{equation}
This ``exact'' sparsity condition has been widely used in estimation; in
this paper we initiate study of its implications for
predictive density estimation.

The minimax K--L risk for estimation over $\Theta$ is given by
%
%e6 #&#
\begin{equation}
\label{eq:mmx-risk-def} R_N(\Theta) = \inf_{\hat p} \sup
_{\thet\in\Theta} \rho( \thet, \hat p),
\end{equation}
where the infimum is taken over \textit{all} measurable predictive
density estimators $\hat p({y}|{x})$.
For comparison, we write
$R_{\mathcal{{E}}}(\Theta) = \inf_{\hat\thet} \sup_\Theta
\rho(
\thet, p[\hat\thet])$ for the minimax risk restricted to the
sub-class $\mathcal{{E}}$ of plug-in or ``estimative'' densities.

To state our main results, henceforth we will assume $v_x =
1$ and introduce the key parameters
%
%e7 #&#
\begin{equation}
\label{eq:oracle} r = v_y/v_x = v_y,\qquad
v_w = \bigl(1+r^{-1}\bigr)^{-1}.
\end{equation}
Here $v_w$ is
%An important role is played by
the ``oracle variance''
% \[
% v_w = (1+r^{-1})^{-1}
% \]
which would be the variance of the UMVUE for~$\thet$, were
\textit{both} $\X$ and $\Y$ observed.

In our asymptotic model, the dimensionality $n \to\infty$ and the
sparsity $s = s_n$ may depend on $n$, but the variance ratio $r$
remains fixed. The notation $a_n \sim b_n$ denotes
$a_n/b_n \to1$ as $n \to\infty$.
%$a_n=b_n(1+\smallo(1))$ as $n \to\infty$.

%th1 #&#
\renewcommand{\theTheoremm}{1\textsc{\alph{Theoremm}}}
\setcounter{Theoremm}{0}
\begin{Theoremm}
 \label{th:mmx-risk}
 Fix $r \in(0,\infty)$.
If $\eta_n = s_n/n \to0$, then
%
%e8 #&#
\begin{equation}
\label{eq:minimax-risk} R_N\bigl(\Theta_n[s_n]\bigr)
\sim\frac{1}{1+r} s_n \log(n/s_n) = \frac{1}{1+r}
n \eta_n \log\eta_n^{-1}.
\end{equation}
\end{Theoremm}

The minimax risk is proportional to the sparsity $s_n$, with a
logarithmic penalty factor.
The case where $s_n \equiv s$ remains constant is included.
The expression is quite analogous to that obtained for point
estimation with quadratic loss, namely,
$2 s_n \log(n/s_n)$
[\citet{Donoho94b,Donoho92} and \citet{Johnstone-book}, Chapter~8.8,
hereafter cited as \citet{Johnstone-book}].
However, we shall see that quite different
phenomena emerge in the predictive density setting.

Indeed, the future-to-past variance ratio $r$ is an important parameter
of the
predictive estimation problem. The minimax risk increases as $r$
decreases: we need to estimate the future observation density based
on increasingly noisy past observations
(in relative terms, $r = v_y/v_x$), and so the difficulty of the
density estimation problem increases.
However, the \textit{rate} of convergence with $n$ in
(\ref{eq:minimax-risk}) does not depend on $r$, and so exact
determination of the constants is needed to show the role of $r$ in
this prediction problem.

The inefficiency of plug-in estimators is an immediate consequence of
Theorem~\ref{th:mmx-risk}.
Let $q(\thet, \hat\thet) = E \| \hat\thet(\X) - \thet\|^2$ denote
the risk of point estimator $\hat\thet$ under \textit{squared-error} loss.
It is straightforward to show for a plug-in density estimate
$p[\hat\thet]$ that
$\rho( \thet, p[\hat\thet]) = q( \thet, \hat\thet)/(2r)$.
Hence, from the point estimation minimax risk just cited,
\[
R_{{\mathcal{E}}}\bigl(\Theta_n[s_n]\bigr) \sim
\frac{1}{r} s_n \log(n/s_n) \sim \biggl(1 +
\frac{1}{r} \biggr) R_N\bigl(\Theta_n[s_n]
\bigr).
\]
The inefficiency of plug-in estimators thus equals the oracle
precision,
\[
1/v_w = 1 + 1/r,
\]
and becomes arbitrarily large as the variance ratio $r \to0$.

We turn now to the asymptotically least favorable priors and optimal
estimators in Theorem~\ref{th:mmx-risk}.
Let $\delta_\lambda$ denote unit point mass at $\lambda$ and
%
%e9 #&#
\begin{equation}
\label{eq:two-point} \pi[\eta, \lambda] = (1-\eta) \delta_0 + \eta
\delta_\lambda
\end{equation}
be a univariate two-point prior: this is a sparse prior when $\eta$ is
small and $\lambda$ large.
Let
%
%e10 #&#
\begin{equation}
\label{eq:lambdadefs} \lambda_e = \sqrt{2 \log
\eta_n^{-1}(1-\eta_n)},\qquad \lambda_f =
\sqrt{v_w} \lambda_e.
\end{equation}
In point estimation based on $\X$, we recall that $\lambda_e$ is
essentially the threshold of detectability corresponding to sparsity
$\eta_n = s_n/n$.
Although $\Y$ is not yet observed, we will
see that in the prediction setting
the UMVUE scaled threshold $\lambda_f < \lambda_e$ plays a partly
analogous role.

Build a sparse high-dimensional prior from i.i.d. draws:
%
%e11 #&#
\begin{equation}
\label{eq:iid-prior} \pi_n^\mathrm{ IID} (d \thet) = \prod
_{i=1}^n \pi[\eta_n,
\lambda_f] (d \theta_i).
\end{equation}
If the sparsity $s_n$ increases without bound with $n$, then this
i.i.d. prior with scale $\lambda_f$ is asymptotically least favorable:

%th1 #&#
\begin{Theoremm}
 \label{th:least-fav}
If $s_n \to\infty$ and $s_n/n \to0$, then
\[
B\bigl(\pi_n^\mathrm{ IID}\bigr) = R_N\bigl(
\Theta_n[s_n]\bigr) \cdot\bigl(1 + o(1)\bigr).
% \mbox{ and
% }\pi_n^{IID}(\Theta[s_n]) \to1 \mbox{ as } n \to\infty.
\]
\end{Theoremm}

The assumption that $s_n \to\infty$ ensures that $\pi_n^\mathrm{ IID}$
concentrates on $\Theta[s_n]$, namely, that
$\pi_n^\mathrm{ IID}(\Theta[s_n]) \to1$ as $n \to\infty$.
This hypothesis
is not needed for Theorem~\ref{th:mmx-risk};
% \ref{th:mmx-risk};
indeed, a sparse prior
built from ``independent blocks'' is asymptotically least favorable
assuming only $s_n/n \to0$. This more elaborate prior is described in
Section~\ref{sec-multivariate}.

% Under very high sparsity when $s_n \not\to\infty$ as $n \to
%can be constructed by independent blocks of sparse prior. We divide $
%[n/s_n]$. Let $\lambda_m = (2 \log m)^{1/2}$ and $\tau_m =
% \lambda_m - \log\lambda_m$. We choose $\nu_m = \sqrt{v_w} \tau_m$ and
% on each of the blocks, we use a single spike prior $\pi_S(\nu_m; m)$.
% The $m$-dimensional prior $\pi_S(\nu_m; m)$ has $m$ equally likely
%support points each of which has only one nonzero co-ordinate of
%length $\nu_m$. The product prior $\pi_n^{IB}$ obtained by making the
%$s_n$ blocks independent:
% \begin{equation}
% \label{eq:spike-prior}
% \pi_n^\mathrm{ IB} (d \thet)
% = \prod_{i=1}^{s_n} \pi_S(\nu_m; m) (d \theta[(i-1)m+1:i m])
% \end{equation}
% is asymptotically least favorable even under very high sparsity.
% \begin{theorem}
% \label{th:least-fav-finite-s}
% If $s_n/n \to0$, then
% \[
% B(\pi_n^\mathrm{ IB}) = R_N(\Theta_n[s_n]) \cdot(1 + o(1)).
% \]
% and $\pi_n^\mathrm{ IB}$ is an asymptotically least favorable prior in $
% \end{theorem}

Some of the novel aspects of the predictive density estimation problem
appear in the description of optimal \textit{estimators}, that is,
ones that
asymptotically attain the
minimax bound in Theorem~\ref{th:mmx-risk}.
In point estimation, the simplest asymptotically minimax rule for
sparsity $s_n$ is given by co-ordinatewise hard thresholding
$\hat\theta_i({x}) = x_i I\{ |x_i| \geq\lambda_e \}$.
For prediction, we consider the following class of univariate density
estimators as analogs of hard thresholding:
%
%e12 #&#
\begin{equation}
\label{eq:pred-thresh} \hat p_T(y_1|x_1) = \cases{
\hat p_\pi(y_1|x_1), &\quad $\mbox{if }
|x_1| \leq\lambda_e,$ \vspace *{2pt}
\cr
\hat
p_U(y_1|x_1), & \quad $\mbox{if } |x_1|
> \lambda_e$.}
\end{equation}
%
%Note that univariate density estimates notationally differ from their
%corresponding multivariate versions through the use of dimensions for
%the observations only.
The univariate density estimates are combined to form a multivariate
predictive density estimate
via a product rule
%
%e13 #&#
\begin{equation}
\label{eq:product} \hat p_T({y}|{x}) = \prod
_{i=1}^n \hat p_T(y_i|x_i).
\end{equation}

The threshold $\lambda_e$ in \eqref{eq:pred-thresh}
is that corresponding to estimation based on
$\X$ at sparsity $\eta_n = s_n/n$.
Above the threshold, the uniform prior predictive density
$\hat p_U$ corresponds to the (unbiased) MLE.
Below threshold, we shall need the flexibility of the Bayes predictive
density (\ref{eq:bpd}).
Indeed, as explained in Section~\ref{sec-up-bound}, it does not
suffice to use $\pi=
\delta_0$, point mass at $0$, which would be the predictive analog of
thresholding to zero in point estimation.

Instead, we use a sparse univariate cluster prior $\pi=
\pi_{\mathrm{CL}}[\eta, r]$ given by
%
%e14 #&#
\begin{equation}
\label{eq:cluster} \pi= (1-\eta) \delta_0 + \frac{\eta}{2K} \sum
_{k=1}^K (\delta_{\mu_k} +
\delta_{-\mu_k}).
\end{equation}
The points $\mu_k = \mu_k(r)$ for $k = 1, \ldots, K$ are
geometrically spaced to cover
an interval $[\nu_\eta, \lambda_e + a]$ containing $[\lambda_f,
\lambda_e]$, as described in more detail below.
The key point is that it is necessary to ``diversify'' the predictive
risk by introducing prior support points to cover $[-\lambda_e,
-\lambda_f] \cup[\lambda_f,
\lambda_e]$.

More specifically, for a parameter $a = a_\eta$ given below, let
$\mu_\eta$ be the positive root of the overshoot equation
%
%e15 #&#
\begin{equation}
\label{eq:overshoot} \mu^2 + 2a \mu= \lambda_e^2,
\end{equation}
that occurs in sparse minimax point estimation
[e.g., % [J13, (8.48)],
\citet{Johnstone-book}, equation~(8.48)],
and then set
$\nu_\eta= \sqrt{v_w} \mu_\eta$: since
$\mu_\eta< \lambda_e$, we have $\nu_\eta< \lambda_f$.
The support points
%
%e16 #&#
\begin{equation}
\label{eq:support} \mu_1 = \nu_\eta, \qquad\mu_{k+1} =
(1+2r)^k \nu_\eta,\qquad k \geq1,
\end{equation}
with $K = \max\{ k \dvtx\mu_k \leq\lambda_e + a\}$.
We choose $a_\eta= \sqrt{2 \log\lambda_f}$.

%th1 #&#
\begin{Theoremm}
 \label{th:asy-mmx-rule}
Assume $\eta_n = s_n/n \to0$.
Let $\hat p_{T,\mathrm{CL}}(\y|\x)$ be the product predictive threshold
estimator defined by (\ref{eq:pred-thresh}) and (\ref{eq:product})
using the cluster prior $\pi_{\mathrm{CL}}[\eta_n,r]$.
Then $\hat p_{T,\mathrm{CL}}$ is asymptotically minimax:
\[
\max_{\Theta_n[s_n]} \rho(\thet, \hat p_{T,\mathrm{CL}}) =
R_N\bigl(\Theta_n[s_n]\bigr) \bigl(1+o(1)
\bigr).
\]
\end{Theoremm}

Note that the number of positive support points in the cluster prior $K
= K_\eta$ increases as $r$
decreases. For any fixed $\eta$, the cluster prior contains
in total $(2 K_\eta+1)$ support points.
Also, for any fixed $r \in(0,\infty)$ as $\eta\to0$, we have
\[
K(r)=\lim_{\eta\to0} K_{\eta} = \biggl\lfloor
\frac{\log(1+r^{-1})}{2 \log(1+2r)} \biggr\rfloor.
\]
Thus, $K(r)$ is a piecewise constant, right continuous function with
jumps as shown in Table~\ref{table:support.points}.
%
%t1 #&#
\begin{table}
\caption{Number $K(r)$ of positive support points in the cluster
prior $\pi_{\mathrm{CL}}[\eta,r]$ as $r$ varies}\label{table:support.points}\label{table1}
\begin{tabular*}{\textwidth}{@{\extracolsep{\fill}}lccccccc@{}}
\hline
$\bolds{r}$ & \textbf{0.1073} & \textbf{0.1235} &
\textbf{0.1465} & \textbf{0.1826} &
\textbf{0.2485} & \textbf{0.4196} & $\bolds{>\!0.4196}$
\\
\hline
$K(r)$ & 7 & 6 & 5 & 4 & 3 & 2 & 1\\
\hline
\end{tabular*}
\end{table}

The results presented above assume $v_x=1$. These results can be
easily extended to the general case by noting that the minimax risk
remains invariant and the scale of past observations and parameter is
divided by $\sqrt{v_x}$.

%s1.2 #&#
\subsection{Background and previous work}\label{sec:backgr-prev-work}

The relative entropy predictive risk $\rho( \thet, \hat p)$
%$\bm{\rho} \big( \thetab, \phat\big)$
measures the exponential rate of divergence of the joint
likelihood ratio over a large number of independent trials
[\citet{Larimore83}]. The minimal predictive risk estimate maximizes the
expected growth rate in repeated investment scenarios [\citet
{Cover-book}, Chapters~6, 15]. In data compression, $L(\thet, \hat
p(\cdot|{x}))$
%$\loss\big( \thetab, \phat\big( \cdot\big|\x) \big)$
reflects the excess average
code length that we need if we use the conditional density estimate
$\phat$ instead of the true density to construct a uniquely decodable
code for the data $\Y$ given the past ${x}$ [\citet{Mcmillan56}].
Following \citet{Bell80}, $\ell_0$-constrained minimax optimal
predictive density estimates in on our model
% with $\ell_0$ constraint on the location structure,
can be used for construction of optimal predictive schemes for
gambling, sports betting, portfolio selection and sparse coding
[\citet{Mukherjee-thesis}, Chapter~1.3].

\citet{Aitchison75,Murray77} and \citet{Ng80} showed that in most
parametric models
there exist Bayes predictive density estimates which are decision
theoretically better than the maximum likelihood plug-in estimate. An
important issue in predictive inference has always been to compare the
performance of the class $\mathcal{E}$ of point estimation (PE) based
plug-in density estimates [\citet{Nielsen96}] with that of the optimal
predictive density estimate.
%In fixed dimensional parametric spaces,
In parameter spaces of fixed dimension,
large sample attributes of the predictive risk of efficient plug-in
and Bayes density estimates have been studied by \citet{Komaki96},
\citet{Hartigan98} and \citet{Aslan06}.

The high-dimensional predictive density estimation problem studied in
this paper is relevant to a number of contemporary applications,
including data compression, sequential investment with side
information and sports betting (SM).

\textit{Analogy with point estimation.} Decision theoretic parallels
between predictive density estimation under Kullback--Leibler loss and
point estimation under quadratic loss have been explored in
our Gaussian model by
%$\textbf{M.2}$
\citet{Komaki04,George06,Ghosh08,Xu12} and \citet{George12}.
For \textit{unconstrained} parameter spaces $\Theta= \mathbb{R}^n$,
fundamental ideas
%techniques and results
in Gaussian point estimation theory
can be extended to yield optimal predictive density estimates
[\citet{Komaki01,Brown08,Fourdrinier11}].
For \textit{ellipsoids},
\citet{Xu10} established an analog of the theorem of \citet{Pinsker80}
% extended the parallels between the two regimes
by proving that the class of all linear predictive density estimates
[see \eqref{eq:linear_estimates}]
is minimax optimal.

For \textit{sparse estimation}, instead of parallels, we found
contrasts.
%Unlike point estimation, the
Minimax risks in
the predictive density problem depend on $r$,
but this dependence is not emphasized in the admissibility results in
unrestricted spaces.
As we have seen, under sparsity
construction of optimal minimax estimators requires the notion of
\emph{diversification} of the future risk over the interval
$[\lambda_f, \lambda_e]$ in a way strongly dependent on $r$.
% Under sparsity constraints
Thus, efficiency of the prediction schemes depend on
careful calibration of the sparsity adjustment and the risk
diversification mechanisms.

%s1.3 #&#
\subsection{Further results}
\label{sec:further-results}
\textit{Other classes of estimators.}
The class of \textit{linear} estimates $\mathcal{L}$ are Bayes rules
based on conjugate product normal priors. The resulting
estimators
%
%e17 #&#
\begin{equation}
\label{eq:linear_estimates} \phat_{L,\alpha} = \prod_{i=1}^n
N(\alpha_i X_i, \alpha_i + r),\qquad
\alpha_i \in[0,1],
\end{equation}
are still Gaussian but have larger variance than the future density
$p(y|\theta,r)=\phi(y|\thetab,r)$. We choose the name ``linear'' because
the conjugate prior implies linearity of the posterior mean in $X$.

The class $\mathcal{G}$ contains all product \emph{Gaussian density
estimates} $p[\hat\thet, \hat d]$
%$\phat_G[ \bm{\that}, \dhatb]$
$= \prod_{i=1}^n N(\that_i, \dhat_i)$.
Clearly, $\mathcal{G}$ contains both $\mathcal{L}$ and
$\mathcal{E}$, the latter introduced after
\eqref{eq:mmx-risk-def}.
The minimax risks $R_{\mathcal{L}}(\Theta)$ and
$R_{\mathcal{G}}(\Theta)$ are defined by restricting the infimum in~(\ref{eq:mmx-risk-def}) to $\mathcal{L}$ and $\mathcal{G}$,
respectively.

We have seen after Theorem~\ref{th:mmx-risk} that
$ R_{\mathcal{E}}(\Theta_n[s_n])
\sim(1 + r^{-1}) R_N(\Theta_n[s_n])$. It turns out that extending
$\mathcal{E}$ to $\mathcal{G}$ does not help, while, as is typical for
sparse estimation, the class of linear estimators $\mathcal{L}$
performs very poorly.

%pr1 #&#
\begin{proposition}
\label{prop:further-results-1}
Fix $r \in(0,\infty)$. If
%$s_n \to\infty$ \texttt{[??]} and
$s_n/n \to0$, then
\begin{eqnarray*}
 R_{\mathcal{L}}\bigl(\Theta_n[s_n]\bigr) &=& (n/2) \log
\bigl(1+r^{-1}\bigr),
\\
 R_{\mathcal{L}}\bigl(\Theta_n[s_n]\bigr) /
R_N\bigl(\Theta_n[s_n]\bigr)& \to& \infty\quad
\mbox{and}
\\
R_{\mathcal{G}}\bigl(\Theta_n[s_n]\bigr) &\sim&
R_{\mathcal{E}}\bigl(\Theta_n[s_n]\bigr).
\end{eqnarray*}
\end{proposition}

\textit{Univariate prediction problem.}
The product structure of our high-\break dimensional model
(\ref{eq:modelM2}), estimators (\ref{eq:product}) and priors
(\ref{eq:iid-prior}), along with concentration of measure, implies
that many aspects of our multivariate results can be understood and
proved through an associated univariate prediction problem.

In the univariate setting, assume that the past observation
$X|\theta\sim N(\theta,1)$ and the future observation
$Y|\theta\sim N(\theta,r)$.
Assume that $X$ and $Y$ are independent given $\theta$.
%independently of $X$.
In addition, suppose that $\theta$ is random with distribution $\pi(d
\theta)$, assumed to belong to
%
%e18 #&#
\begin{equation}
\label{eq:m0} \mathfrak{m}(\eta)=\bigl\{\pi\in\mathcal{P}(\RR)\dvtx\pi(\theta
\neq0) \leq\eta\bigr\},
\end{equation}
where $\mathcal{P}(\RR)$ is the collection of all probability
measures in $\RR$.

A predictive density estimator $\hat p(y|x)$ is evaluated through its
integrated risk $B(\pi, \hat p)$ defined at (\ref{eq:integ-risk}).
The minimax risk for this univariate prediction problem is given by
%
%e19 #&#
\begin{equation}
\label{eq:beta-eta} \beta(\eta,r):=\inf_{\phat} \sup
_{\pi\in\mathfrak{m}(\eta)} B(\pi,\phat),
\end{equation}
and we study sparsity through the asymptotic
regime $\eta\to0$. Recall definition (\ref{eq:lambdadefs}) of the
scaled threshold $\lambda_f = \lambda_{f,\eta}$.

%th2 #&#
\begin{theorem}
\label{th:univ-mmx-result}
{\textsc{a}.} Fix $r \in(0, \infty)$. As $\eta\to0$,
%
%e20 #&#
\begin{equation}
\label{eq:beta-rate} \beta(\eta,r) = \frac{1}{2r} \eta\lambda_f^2
\bigl(1 + o(1)\bigr).
\end{equation}

{\textsc{b}.} An asymptotically least favorable prior is given by the two-point
distribution $\pi[\eta, \lambda_f(\eta)]$ of (\ref{eq:two-point}).

{\textsc{c}.} An asymptotically minimax estimator is given by the thresholding
construction (\ref{eq:pred-thresh}) combined with sparse univariate
cluster prior $\pi= \pi_{\mathrm{CL}}[\eta,r]$ defined at (\ref{eq:cluster}).
\end{theorem}

%%
%The asymptotic minimax rules $\hat p_T$ described in Theorems
%thresholding.
%It would be desirable to construct a prior $\pi$ for which the Bayes
%predictive density $\hat p_\pi$ in (\ref{eq:bpd}) is itself
%asymptotically minimax, without any use of the discontinuous
%thresholding
%operation.
%
%Consider, then, a symmetric univariate prior $\pi_\infty[\eta,r]$
%whose support consists of the origin and infinite number of
%equidistant clusters each containing $2K$ points
%%(defined before in equation~[\ref{cluster.prior.defn}] )
%in the same spatial alignment as for $\pi_{CL}[\eta,r]$:
% \pi_\infty[\eta,r]
% = (1-\eta) \delta_0 + \frac{1-\eta}{2}
% \sum_{j=0}^\infty\eta^{j+1}
% \sum_{k=1}^K q_k (\delta_{\mu_{jk}} + \delta_{-\mu_{jk}}),
%where
%$\mu_{jk} = j \lambda_e + \mu_k$ and for $k = 2, \ldots K$ and $\gamma
%= \log\eta^{-1}$, we have
%$ q_k = \gamma^{-k}$ and $q_1 = 1 - \sum_2^K q_k$.
%% \[
%% q_k = \gamma^{-k}, q_1 = 1 - \sum_2^K q_k.
%% \]
%
%Based on $\pi_\infty[\eta_n,r]$ we may construct a multivariate prior
%$\pi_{n,\infty}^\mathrm{ IID}$ using (\ref{eq:iid-prior}) which will not
%only be least favorable but also yield a minimax optimal density
%estimate.
%
%Under the conditions of Theorem~\ref{th:mmx-risk} for any fixed $r \in
%(0,\infty]$, the proper prior distribution
%$\pi_{n,\infty}^\mathrm{ IID}$
%%$\pi[n,s,r,\Inf]$
%is asymptotically least favorable and its corresponding Bayes
%predictive
%density is asymptotically minimax optimal.

%s1.4 #&#
\subsection{Organization of the paper}
The main results of the paper are \textit{multivariate}, Theorems
\ref{th:mmx-risk}, \ref{th:least-fav} and \ref{th:asy-mmx-rule}. However, the main technical
issues in the proofs are best handled in the \textit{univariate}
setting of Theorem~\ref{th:univ-mmx-result}, whose parts \textsc{a},
\textsc{b} and \textsc{c} correspond to Theorems
\ref{th:mmx-risk}, \ref{th:least-fav} and \ref{th:asy-mmx-rule}, respectively.
Section~\ref{sec-overview} has an overview: it first reviews some connections between
the multivariate and univariate settings, then gives
heuristic derivations for the lower and upper bounds of univariate
Theorem~\ref{th:univ-mmx-result}.
% The proof of the main results and their implications are described
%first in an overview fashion in Section~\ref{sec-overview}.
Section~\ref{sec-lower-bound} and Section~\ref{sec-up-bound},
respectively, contain the technical proofs for the lower and upper
bound on the univariate minimax risk,
Theorem \ref{th:univ-mmx-result}\textsc{b} and \ref{th:univ-mmx-result}\textsc{c}, respectively.
Together, they complete the proof of Theorem~\ref{th:univ-mmx-result}.
Proofs of the multivariate results in Theorems \ref{th:mmx-risk},
\ref{th:least-fav} and \ref{th:asy-mmx-rule}
% Theorems~\ref{th:mmx-risk},
% \ref{th:least-fav},
% %\ref{th:least-fav-finite-s} and
% and \ref{th:asy-mmx-rule}
are completed in Section~\ref{sec-multivariate}. This section also
contains a heuristic proof of Proposition~\ref{prop:further-results-1}
whose rigorous proof is presented in the supplementary material
[\citet{Supplementary}].

\textit{Glossary.} [The notation \eqref{eq:mmx-risk-def}$+$2 refers to
text 2 lines after equation \eqref{eq:mmx-risk-def}].

Estimators: Bayes $\hat p_\pi$ \eqref{eq:bpd},
Uniform prior $\hat p_U$ \eqref{eq:bpd}$+$2,
Threshold $\hat p_T$ \eqref{eq:pred-thresh},
Multivariate product $\hat p(y|x)$ \eqref{eq:product};
Univariate $\hat p(y_1|x_1)$.

Classes of estimators and multivariate minimax risks:
all nonlinear $N, R_N$ \eqref{eq:mmx-risk-def},
estimative $\mathcal{E}, R_{\mathcal{E}}$ \eqref{eq:mmx-risk-def}$+$2,
``linear'' $\mathcal{L}, R_{\mathcal{L}}$ \eqref{eq:linear_estimates},
Gaussian $\mathcal{G},R_{\mathcal{G}}$ \eqref
{eq:linear_estimates}$+$4.

Univariate minimax risk: $\beta$ \eqref{eq:beta-eta}.

Parameter spaces: multivariate $\Theta_n[s]$ \eqref{eq:theta-n-s};
univariate $\mathfrak{m}(\eta)$ \eqref{eq:m0}.

Priors: Univariate: two point $\pi[\eta,\lambda]$ \eqref{eq:two-point},
cluster $\pi_{\mathrm{CL}}[\eta,r]$ \eqref{eq:cluster},
Multivariate: $\pi_n^\mathrm{ IID}$ \eqref{eq:iid-prior}.

Parameters: variance ratio $r = v_y/v_x$, oracle variance
$v_w$ \eqref{eq:oracle},
sparsity $\eta$ \eqref{eq:two-point},
thresholds $\lambda_e, \lambda_f$ \eqref{eq:lambdadefs},
cluster prior: overshoot $a$ \eqref{eq:overshoot}, $\nu_\eta$
\eqref{eq:overshoot}$+$2.

%%% Local Variables:
%%% mode: latex
%%% TeX-master: "main"
%%% End:

%s2 #&#
\section{Proof overview and interpretation}\label{sec-overview}

%s2.1 #&#
\subsection{Connections between multivariate and univariate settings}
\label{sec:conn-betw-mult}

Many aspects of the multivariate theorem may be understood, and in
part proved, through a discussion of the univariate prediction problem
of Theorem~\ref{th:univ-mmx-result}.
An obvious connection between the univariate and multivariate
approaches runs as follows: suppose that a multivariate predictive
estimator is built as a product of univariate components
%
%e21 #&#
\begin{equation}
\label{eq:prod-rule} \hat p(y|x) = \prod_{i=1}^n
\hat p_1(y_i|x_i).
\end{equation}
Suppose also that to a vector $\theta= (\theta_i)$ we associate a
univariate (discrete) distribution
$\pi_n^e = n^{-1} \sum_{i=1}^n \delta_{\theta_i}$.
Since the true multivariate future density $p(Y| \theta, r)$ is also a
product of univariate
components, it is then readily seen that the multivariate and univariate
Bayes K--L risks are related by
%
%e22 #&#
\begin{equation}
\label{eq:prod-risk} \rho( \theta, \hat p) = \sum_{i=1}^n
\rho(\theta_i, \hat p_1) = n B\bigl(
\pi_n^e, \hat p_1\bigr).
\end{equation}
The sparsity condition $\Theta_n[s_n]$ in the multivariate problem
corresponds to requiring that the prior $\pi= \pi_n^e$ in the
univariate problem satisfies
\[
\pi\{ \theta_1 \neq0 \} \leq s_n/n=\eta_n,
\]
and thus belongs to the class $\mathfrak{m}(\eta)$ defined in \eqref{eq:m0}.
Next, we outline the minimax risk calculations for the sparse
predictive density estimation problem.

%
% $1^\circ$.
As a first illustration, to which we return later, consider
the maximum risk of a product rule over $\Theta_n[s_n]$: using
\eqref{eq:prod-risk} and \eqref{eq:integ-risk}, we have
%
%e23 #&#
\begin{equation}
\label{eq:multi-sup} \sup_{\Theta_n[s_n]} \rho( \theta, \hat p) = n \Bigl[ (1-
\eta_n) \rho(0,\hat p_1) + \eta_n \sup
_{\theta
\in\mR} \rho(\theta, \hat p_1) \Bigr].
\end{equation}
In the univariate problem, using $\hat p_1$, we have the somewhat parallel
bound
%
%e24 #&#
\begin{equation}
\label{eq:bayes-max} \sup_{\me} B(\pi,\hat p) = (1-\eta) \rho(0,\hat
p_1) + \eta\sup_{\theta
\in\mR} \rho(\theta, \hat
p_1).
\end{equation}
Consequently, a careful study of the two univariate quantities
%
%e25 #&#
\begin{eqnarray}
\label{eq:key-parts}&&\mbox{risk at zero:}\qquad  \rho(0, \hat p_1),
\nonumber
\\[-8pt]
\\[-8pt]
\nonumber
&&\mbox{maximum risk:} \qquad \sup_\theta\rho(\theta,
\hat p_1)
\end{eqnarray}
is basic for upper bounds for both univariate and multivariate
cases.

%
% $2^\circ$.

%s2.2 #&#
\subsection{Theorem~\texorpdfstring{\protect\ref{th:univ-mmx-result}{\normalfont{\textsc{b}}}}{2B}: Univariate lower bound heuristics}
\label{sec:theor-2b:-univ}

To understand the apperance of $\lambda_f^2$ in the minimax
risks, we turn to a heuristic discussion of the lower bound, first in
the univariate case.

We use the two point priors \eqref{eq:two-point} and the definition
\eqref{eq:beta-eta}:
% \begin{equation}
% \label{eq:basic1}
% \begin{split}
% \beta(\eta,r)
% & \geq B(\pi[\eta,\lambda]) \\
% & = (1-\eta) \rho(0,\hat p_\pi) + \eta\rho(\lambda, \hat p_\pi).
% \end{split}
% \end{equation}
%
%e26 #&#
\begin{equation}
\label{eq:basic1} \beta(\eta,r) \geq B\bigl(\pi[\eta,\lambda]\bigr) = (1-\eta)
\rho(0,\hat p_\pi) + \eta\rho(\lambda, \hat p_\pi) \geq
\eta\rho(\lambda, \hat p_\pi),
\end{equation}
and look for a good bound for $\rho(\lambda, \hat p_\pi)$ for a suitable
choice of $\lambda$.

The key is a mixture representation for predictive risk of a Bayes
estimator in terms of quadratic risk, where the weighted mixture is
over noise levels $v \in[v_w, 1]$, with $v_w$ being the oracle
variance, \eqref{eq:oracle}.
\citet{Brown08}, Theorem~1, show that the predictive risk of
the Bayes predictive density estimate $\hat p_{\pi}$ is
%
%e27 #&#
\begin{equation}
\label{eq:mixture} \rho(\theta, \hat p_\pi) = \frac{1}{2} \int
_{v_w}^1 q (\theta, \hat\theta_{\pi,v}; v)
\frac{dv}{v^2},
\end{equation}
where $q (\theta, \hat\theta_{\pi,v}; v) =
E_\theta[\hat\theta_{\pi,v}(W) - \theta]^2$
is the quadratic risk of the Bayes location estimate
$\hat\theta_{\pi,v}$ for prior $\pi$ when $W \sim N(\theta,v)$.
In point estimation with quadratic loss, it is known
[\citet{Johnstone-book}, Chapter~8],
that as $\eta\to0$ an approximately least favorable prior in the
class $\me$ is given, for noise level $v=1$, by the sparse two-point
prior $\pi[\eta, \lambda_e(\eta)]$ defined in \eqref{eq:two-point} and
$\lambda_e(\eta)=\sqrt{2 \log\eta^{-1}(1-\eta) }$. This prior has the
remarkable property that points $\theta\leq\lambda_e$ are
``invisible'' in the sense that even when $\theta$ is true, the Bayes
estimator $\hat\theta_\pi= \hat\theta_{\pi,1}$ effectively
estimates $0$ rather than
$\theta$ and so makes a mean squared error
%
%e28 #&#
\begin{equation}
\label{eq:invisible} q(\theta, \hat\theta_\pi; 1) \sim\theta^2
\qquad\mbox{for } 0 \leq\theta\leq\lambda_e.
\end{equation}

% The noise level $v$ in mixture \eqref{eq:mixture} varies in the range
% $v_w \leq v \leq1$.
Two issues arise as the noise level $v$ varies.
First, the region of invisibility will scale, becoming $0 \leq\theta
\leq\sqrt v \lambda_e$ at scale $v$.
As $v$ varies in $[v_w,1]$, the intersection of all regions of
invisibility will be $0 \leq\theta\leq\sqrt v_w \lambda_e =
\lambda_f$ as defined at \eqref{eq:lambdadefs}.
The second issue is that for a given prior $\pi$ and predictive Bayes
rule $\hat p_\pi$ in \eqref{eq:mixture}, the Bayes rules
$\hat\theta_{\pi,v}$ vary with $v$.
We return to this second point in the next section; for now we can
hope that for all $v \in[v_w,1]$,
%
%e29 #&#
\begin{equation}
\label{eq:invis-risk} q(\lambda_f, \hat\theta_{\pi,v}; v) \gtrsim
\lambda_f^2,
\end{equation}
and so, from mixture representation \eqref{eq:mixture},
\[
\rho( \lambda_f, \hat p_\pi) \gtrsim\frac{\lambda_f^2}{2}
\int_{v_w}^1 \frac{dv}{v^2} =
\frac{\lambda_f^2}{2r},
\]
since the integral evaluates to $v_w^{-1} -1 = r^{-1}$.
From this we can conjecture that for
$\pi= \pi[\eta, \lambda_f] \in\me$,
%
%e30 #&#
\begin{equation}
\label{eq:lower} B(\pi) > \eta\rho(\lambda_f, \hat p_\pi)
\gtrsim\eta\frac{\lambda_f^2}{2r}.
\end{equation}
A full proof, with slightly modified definitions, is given in
Section~\ref{sec-lower-bound}.

\subsection{Theorem~\texorpdfstring{\protect\ref{th:univ-mmx-result}{\normalfont{\textsc{c}}}}{2C}: Univariate upper bound heuristics}
\label{sec:theor-2c:-univ}

We now turn to a heuristic discussion of constructing
a density estimate to show that the lower bound
\eqref{eq:lower} is asymptotically correct.
Pursuing the analogy with point estimation, we know that
in that setting optimal
estimators can be found within the family of hard thresholding rules
$\hat\theta(x) = x I \{ |x| > \lambda\}$.
The natural analog for predictive density estimation would have the
form
%
%e31 #&#
\begin{equation}
\label{eq:p-hat-T} \hat p_{T,\pi_0}[\lambda](y|x) = \cases{ \hat
p_U (y|x), &\quad $|x| > \lambda,$ \vspace*{2pt}
\cr
\hat
p_{\pi_0}(y|x), &\quad $|x| \leq\lambda$.}
\end{equation}
To see this, note that $\hat p_U$ is the predictive Bayes rule
corresponding to the uniform prior $\pi( d\theta) = d \theta$, which
leads to the MLE $\hat\theta(x) = x$ in point estimation, while
$\hat p_{\pi_0}(y|x)$ denotes the predictive Bayes rule corresponding
to a prior concentrated entirely at $0$, so that
%
%e32 #&#
\begin{equation}
\label{eq:zero-prior} \hat p_{\pi_0}(y|x) = \phi(y|0,r)
\end{equation}
is a normal density with mean zero and variance $r$.

For the upper bound, according to definition \eqref{eq:beta-eta},
we seek an estimator $\hat p_1$ for which
$\sup_{\me} B(\pi, \hat p_1) \sim\eta\lambda_f^2/(2r)$ as $\eta
\to
0$.
In bound \eqref{eq:bayes-max}, the first component
is the risk at zero, $\rho(0, \hat p_1)$,
and it turns out that this determines the possible values of the
threshold $\lambda$ in \eqref{eq:p-hat-T}.
Thus, in order that
\[
\rho\bigl(0, \hat p_{T,\pi_0} [\lambda] \bigr) = o\bigl( \eta
\lambda_f^2\bigr),
\]
it follows [see \eqref{eq:min_threshold_size}] that the threshold
$\lambda$ should be chosen
as $\lambda= \lambda_e \sim (2 \log\eta^{-1})^{1/2}$ and not
smaller.

% Turning to the second key component in
% \eqref{eq:key-parts}, we seek an estimator $\hat p_1$ such
% that
Turning to the second part of
\eqref{eq:key-parts}, we seek an estimator $\hat p_1$ with
%
%e33 #&#
\begin{equation}
\label{eq:unif-bd} \sup_\theta\rho(\theta, \hat p_1) =
\frac{\lambda_f^2}{2r} \cdot\bigl(1 + o(1)\bigr).
\end{equation}

We first argue that the hard thresholding analog
\eqref{eq:p-hat-T} cannot work.
%For that purpose,
Decompose the predictive risk of a univariate threshold estimator
$\hat p_T$ with threshold $\lambda_e$ into contributions due to $X$
above and below the threshold
% \begin{eqnarray}\label{eq:th-decom}
% \rho(\theta, \hat p_T)
% & = E_\theta L(\theta, \hat p( \cdot| X)) \\
% & = E_\theta[L(\theta, \hat p_U(\cdot|X)), |X| > \lambda_e]
% + E_\theta[L(\theta, \hat p_\pi(\cdot|X)), |X| \leq\lambda_e] \\
% & = \rho_A(\theta) + \rho_B(\theta),
% \end{eqnarray}
%
%e34 #&#
%e35 #&#
%e36 #&#
\begin{eqnarray}
\label{eq:th-decom} \rho(\theta, \hat p_T) & =& E_\theta L\bigl(
\theta, \hat p( \cdot| X)\bigr)\nonumber
\\
& =& E_\theta\bigl[L\bigl(\theta, \hat p_U(\cdot|X)\bigr),
|X| > \lambda_e\bigr] + E_\theta\bigl[L\bigl(\theta, \hat
p_\pi(\cdot|X)\bigr), |X| \leq\lambda_e\bigr]
\\
& =& \rho_A(\theta) + \rho_B(\theta),\nonumber
\end{eqnarray}
say.
With the ``zero prior,'' the K--L loss is just quadratic in $\theta$,
\[
L\bigl(\theta, \hat p_{\pi_0} (Y|X)\bigr) = E_\theta\log
\frac{\phi(Y|\theta,r)}{\phi(Y|0,r)} = \frac{\theta^2}{2r},
\]
and so, in particular, for $\theta\leq\lambda_e$ we see that
%
%e37 #&#
\begin{equation}
\label{eq:quadratic} \rho(\theta, \hat p_{T,\pi_0}) \geq\rho_B(
\theta) \gtrsim\frac{\theta^2}{2r} P_\theta\bigl[|X| \leq\lambda_e\bigr]
\end{equation}
could be as large as $\lambda_e^2/(2r)$, and hence larger
than our target risk $\lambda_f^2/(2r)$.

Bearing in mind the role that two-point priors play in the lower
bound, it is perhaps natural to ask next if the threshold rule $\hat p_{T,\mathrm{LF}}$
with $\pi_0$ in \eqref{eq:p-hat-T} replaced by the (symmetrized)
% symmetric version $\pi_{3}[\eta,\lambda_f]$ of the asymptotically
% least favorable
two-point prior $\pi[\eta, \lambda_f]$ could cut off the
growth of the quadratic $\theta^2/(2r)$ for $|\theta| \geq\lambda_f$.
The $3$-point prior $\pi_{3}[\eta,\lambda_f] \in\me$
%is symmetric prior in $\me$ with $\eta/2$
places probability $\eta/2$ at the two nonzero atoms at
$\pm\lambda_f$. Remarks in Section~\ref{sec-lower-bound}
show that $\pi_{3}[\eta,\lambda_f]$ is also asymptotically least favorable
for the univariate prediction problem as $\eta\to0$.
% and so it is
% interesting to study the role of the symmetric threshold density
% estimate $\hat p_{T,LF}$.
Indeed, it can be shown (see Section~\ref{sec-up-bound}) that for this
prior and for
$\lambda_f \leq|\theta| \leq\lambda_e$,
%
%e38 #&#
\begin{eqnarray}
\label{eq:three-pt-bound} \rho(\theta, \hat p_{T,\mathrm{LF}}) &\sim&\rho_B(
\theta)
\nonumber
\\[-8pt]
\\[-8pt]
\nonumber
&\leq&\frac{1}{2r} \bigl\{ \lambda_f^2 - \bigl(|
\theta| - \lambda_f\bigr)\bigl[(1+2r)\lambda_f - |\theta|
\bigr] \bigr\} + o\bigl(\lambda_f^2\bigr).
\end{eqnarray}
Consequently, the risk bound dips below $\lambda_f^2/(2r)$ for
$\lambda_f
\leq|\theta| \leq(1+2r) \lambda_f$ but increases thereafter. So,
$\hat p_{T,\mathrm{LF}}$ is minimax optimal if $\lambda_e < (1+2r) \lambda_f$,
which occurs if $r$ is sufficiently large, $r > 0.4196$ in
Table~\ref{table:support.points}. However, the upper bound exceeds our
target risk $\lambda_f^2/(2r)$ if $r \leq0.4196$.
Section~S.2 of the supplementary material [\citet
{Supplementary}] shows
rigorously that $\phat_{T,\mathrm{LF}}$ is indeed minimax suboptimal for low
values of $r$.
% As such, we can show that $\phat_{T,LF}$ is indeed minimax
%sub-optimal for low values of $r$. Section~\ref{sec:supp_LF} of the
%Supplementary materials contains the rigorous proof of the
%sub-optimality of $\phat_{T,LF}$ for our sparse univariate prediction
%problem.

%f1 #&#
\begin{figure}

\includegraphics{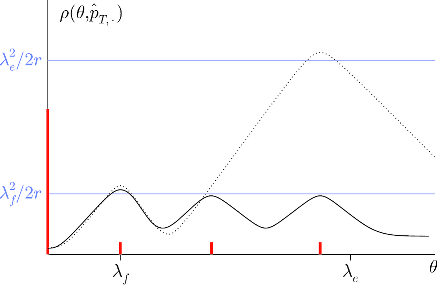}

\caption{Schematic diagram of the risk of univariate threshold
density estimates for $\theta\geq0$.
The dotted line is the risk of density estimator $\hat p_{T,\mathrm{LF}}$ based
on the $3$-point prior $\pi_3[\eta, \lambda_f]$.
% as the parameter $\theta$ varies over the
% nonnegative axis. In dotted lines is the risk of the $3$--point
% prior $\pi_3[\eta,\lambda_f]$ based estimate $\hat p_{T,LF}$.
The addition of appropriately spaced prior mass points (shown in red)
up to $\lambda_e$ pulls down the risk function of the cluster
prior-based density estimate $\hat p_{T,\mathrm{CL}}$ below
$\lambda_f^2/(2r)$ until the effect of thresholding at $\lambda_e$
takes over.}
\label{fig:first1}
\end{figure}

As $\pi_3[\eta,\lambda_f]$ fails to produce minimax optimal density
estimates, the strategy then is to introduce extra support points
$|\mu_k| \leq\lambda_e$
into the prior
% (like the Cluster prior $\pi_{CL}[\eta,r]$ defined in
% \eqref{eq:cluster})
% with the location of the points $ |\mu_k| \leq\lambda_e$
chosen to
``pull down'' the risk $\rho_B(\theta) = E_\theta[L(\theta, \hat
p_\pi( \cdot| X)), |X| \leq\lambda_e]$ below $\lambda_f^2/(2r)$ whenever
it would otherwise exceed this level.
The schematic diagram in Figure~\ref{fig:first1} illustrates this
bounding of the maximum risk. The extra support points added in
$[\lambda_f,\lambda_e]$ and $[-\lambda_e,-\lambda_f]$
%$ between $[-\lambda_e,-\lambda_f] \cup[\lambda_f,\lambda_e]$
distribute the predictive risk across that range---``risk
diversification''---and keep the maximum risk below
$\lambda_f^2/(2r)(1+\smallo(1))$.

To prove that this works, we obtain upper bounds on $\rho_B(\theta)$
for $\hat p_{T,\mathrm{CL}}$ by focusing, when $\theta\in[\mu_k, \mu_{k+1}]$,
only on the prior
support point $\mu_k$. The main inequality is obtained in
\eqref{eq:main-bound}, namely,
\[
\rho_B(\theta) \leq\frac{1}{2r} \Bigl[ \lambda_f^2
+ \min_k q_k(\theta) \Bigr] + o\bigl(
\lambda_f^2\bigr),
\]
where $q_k(\theta)$ is a quadratic polynomial that is $O(\lambda_f)$
on $[\mu_k, \mu_{k+1}]$.
Putting together this and other bounds, we can then finally establish
the uniform bound \eqref{eq:unif-bd}.
The details are in Section~\ref{sec-up-bound}.

%%% Local Variables:
%%% mode: latex
%%% TeX-master: "main"
%%% End:

%s3 #&#
\section{Theorem~\texorpdfstring{\protect\ref{th:univ-mmx-result}{\normalfont{\textsc{b}}}}{2B}: Univariate lower bound proof} \label{sec-lower-bound}

This section is devoted to a proof of the lower bound part of Theorem~\ref{th:univ-mmx-result}.
The heuristic discussion of the last section indicated the importance
of two-point sparse priors and the invisibility property
\eqref{eq:invisible}.
To formulate a precise statement about the upper limit of
invisibility, we start with noise level $1$ and bring in the
positive solution $\mu_\eta$ of the overshoot equation
\eqref{eq:overshoot}, namely, $\mu^2 + 2 a \mu= \lambda_e^2$.
Here the ``overshoot'' parameter $a = a_\eta$ should satisfy
both $a_\eta\to\infty$ and $a_\eta= o(\mu_\eta)$;
we make the specific choice $a_\eta= \sqrt{2 \log
\lambda_{f,\eta}}$.

In preparation for the range of variance scales in mixture
representation \eqref{eq:mixture}, we consider the collection of
two-point priors $\pi[\eta, \mu]$ for $0 \leq\mu\leq\mu_\eta$.
Using a
temporary notation for this section, let $\hat\theta_\mu(x) =
E[\theta|x]$ be the Bayes rule for squared error loss for the prior
$\pi[\eta,\mu]$. The next result shows that when the true parameter
is actually $\mu$, and this nonzero support point $\mu\leq
\mu_\eta$, then the Bayes rule for $\pi[\eta, \mu]$ ``gets it wrong''
by effectively estimating $0$ and making an error of size $\mu^2$,
uniformly in $\mu\leq\mu_\eta$.

%le3 #&#
\begin{lemma}
\label{lem:qr}
There exists $\varepsilon_\eta\searrow0$ as $\eta\to0$ such that
for all $\mu$ in $[0, \mu_\eta]$,
\[
q(\mu, \hat\theta_\mu; 1) \geq\mu^2[1 -
\varepsilon_\eta].
\]
\end{lemma}

\begin{pf}
Using standard calculations for the two-point prior, the Bayes rule
$\hat\theta_\mu= \mu p(\mu|x) = \mu/[1 + m(x)]$, with
%
%e39 #&#
\begin{equation}
\label{eq:lrat} m(x) = \frac{p(0|x)}{p(\mu|x)} =
\frac{1-\eta}{\eta} \frac{\phi(x)}{\phi(x-\mu)}
= \exp\biggl\{ \hf\lambda_e^2 - x \mu+ \hf
\mu^2 \biggr\}.
\end{equation}
Consequently,
\begin{eqnarray*}
q(\mu, \hat\theta_\mu; 1) & =& E_\mu[\hat
\theta_\mu- \mu]^2 = \mu^2 E_\mu
\bigl[\bigl(1+m(X)\bigr)^{-1} - 1\bigr]^2
\\
& = &\mu^2 E_0 \bigl[1 + m^{-1}(\mu+Z)
\bigr]^{-2},
\end{eqnarray*}
where $Z \sim N(0,1)$, and from \eqref{eq:lrat},
$m^{-1}(\mu+z) = \exp\{ \hf(\mu^2 + 2 \mu z - \lambda_e^2) \}$.

Now, using definition \eqref{eq:overshoot} of
$\mu_\eta$, for $0 \leq\mu\leq\mu_\eta$, we have
\[
\mu^2 + 2 \mu z - \lambda_e^2 \leq
\mu_\eta^2 + 2 \mu_\eta z_+ -
\lambda_e^2 = - 2 \mu_\eta(a-z_+),
\]
so that for $0 \leq\mu\leq\mu_\eta$,
\[
\mu^{-2} q(\mu, \hat\theta_\mu; 1) \geq E_0
\bigl\{ \bigl[1 + \exp\bigl(-\mu_\eta(a-Z_+) \bigr)
\bigr]^{-2}, Z < a \bigr\} % \geq E_0 \{ [1 + e^{-2(a-Z_+) \mu_\eta}]^{-2}, Z < a \}
= 1 - \varepsilon_\eta,
\]
say. For each fixed $z$, we have $\mu_\eta(a-z_+) \to\infty$ since
$a \to\infty$, and so from the dominated convergence theorem we
conclude that $\varepsilon(\eta) \to0$.
\end{pf}

With these preparations, we return to the lower bound in the
prediction problem.
As $\eta\to0$, an asymptotically least favorable distribution is
given by a sparse two-point prior with the nonzero support point
scaled using the oracle standard deviation $v_w^{1/2}$. We shall prove
the following:

%le4 #&#
\begin{lemma} \label{lem:lower-bound}
Let $\mu_\eta$ be the positive solution to overshoot equation
\eqref{eq:overshoot} with $a_\eta= \sqrt{2 \log\lambda_{f,\eta}}$.
Set $\nu_\eta= v_w^{1/2} \mu_\eta$ and consider the two-point prior
$\pi[\eta,\nu_\eta]$. Then as $\eta\to0$,
\[
\beta(\eta,r) \geq B\bigl(\pi[\eta, \nu_\eta]\bigr) \geq
\frac{\eta\lambda_f^2}{2r} \bigl(1 + o(1)\bigr).
\]
\end{lemma}

We note here that since $a_\eta= o(\mu_\eta)$, the
overshoot equation implies that
%
%e40 #&#
\begin{equation}
\label{eq:asy-rels} \mu_\eta\sim\lambda_{e,\eta} \quad\mbox{and}\quad
\nu_\eta\sim\lambda_{f,\eta}.
\end{equation}
A stronger conclusion, used in the next section, also follows from the
overshoot equation, namely,
%
%e41 #&#
\begin{equation}
\label{eq:diffbound} \lambda_{f,\eta}^2 - \nu_\eta^2
= v_w\bigl(\lambda_{e,\eta}^2 -
\mu_\eta^2\bigr) = v_w \cdot2 a
\mu_\eta \leq2 a v_w \lambda_{e,\eta} = 2 a \sqrt
v_w \lambda_{f,\eta}.
\end{equation}

\begin{pf*}{Proof of Lemma \ref{lem:lower-bound}}
Recall \eqref{eq:basic1} and \eqref{eq:mixture} in the heuristic
discussion.
We now clarify the dependence on scale $v$ of the Bayes rule
$\hat\theta_{\pi,v}$ in the mixture representation
\eqref{eq:mixture}.
% Suppose that we use a two point prior $\pi[\eta, \lambda]$, compare
% \eqref{eq:two-point}.
Passing from noise level $v$ to noise level $1$ by dividing
parameters and estimates by $v^{1/2}$, we obtain the invariance
relation
\[
q(\theta, \hat\theta_{\pi[\eta, \lambda],v}; v) = v q \bigl( v^{-1/2} \theta,
\hat\theta_{\pi[\eta,
v^{-1/2} \lambda]}; 1 \bigr).
\]
Now set $\theta= \nu_\eta$ and substitute into \eqref{eq:mixture} to
obtain,
for $\pi= \pi[\eta, \nu_\eta]$,
%
%e42 #&#
\begin{equation}
\label{eq:mixture1} \rho(\nu_\eta, \hat p_\pi) =
\frac{1}{2} \int_{v_w}^1 q \bigl(
v^{-1/2} \nu_\eta, \hat\theta_{\pi[\eta,
v^{-1/2} \nu_\eta]}; 1 \bigr)
\frac{dv}{v}.
\end{equation}
Now apply Lemma~\ref{lem:qr} with $\mu= v^{-1/2} \nu_\eta$ being
bounded above by $v_w^{-1/2} \nu_\eta= \mu_\eta$.
For all $v \in[v_w,1]$ we obtain
\[
q\bigl(v^{-1/2}, \hat\theta_{v^{-1/2} \nu_\eta}; 1\bigr) \geq
v^{-1} \nu_\eta^2 [1 - \varepsilon_\eta].
\]
Putting this into the mixture representation, we get
\[
\rho(\nu_\eta, \hat p_\pi) \geq\frac{1}{2}
\nu_\eta^2 [1 - \varepsilon_\eta] \int
_{v_w}^1 \frac{dv}{v^2} = \frac{\nu_\eta^2}{2r} [1
- \varepsilon_\eta].
\]
Taking into account both \eqref{eq:basic1} and \eqref{eq:asy-rels}, we
have established the lemma. %\qedhere
\end{pf*}

Based on the discussion in Section~\ref{sec-overview}, the above
lemma establishes a lower bound on the asymptotic minimax risk
$\beta(\eta,r)$ in Theorem~\ref{th:univ-mmx-result}.
% Following the
% above arguments and the connections with the quadratic risk, it is
% easy to see that
Similarly, the symmetric $3$-point prior
\[
\pi_3[\eta,\nu_{\eta}]=(1-\eta) \delta_0 + (
\eta/2) \{\delta_{\nu_{\eta}} + \delta_{-\nu_{\eta}} \}
\]
will also be
asymptotically least favorable over $\me$ as $\eta\to0$.
% However,
% as in the quadratic loss case, it is not necessary to have a symmetric
% prior to attain the maximum asymptotic Bayes risk.

%%% Local Variables:
%%% mode: latex
%%% TeX-master: "main"
%%% End:

%s4 #&#
\section{Theorem~\texorpdfstring{\protect\ref{th:univ-mmx-result}{\normalfont{\textsc{c}}}}{2C}: Univariate upper bound proof}\label{sec-up-bound}

The upper bound on the predictive minimax risk $\beta(\eta,r)$ is
derived from the upper bound on the maximum Bayes risk of $ \hat
p_{T,\mathrm{CL}}$ over $\me$.
In this section we will prove the following lemma which along with
Lemma~\ref{lem:lower-bound} completes the proof of Theorem~\ref
{th:univ-mmx-result}.

%le5 #&#
\begin{lemma}\label{lem:up-bound-sec4}
For any $r \in(0, \infty)$ we have, as $\eta\to0$,
\[
\sup_{\pi\in\me} B(\pi, \hat p_{T,\mathrm{CL}}) \leq
\frac{\eta\lambda_f^2}{2r} \bigl(1 + o(1)\bigr).
\]
\end{lemma}

We consider a threshold predictive density estimate $\hat p_T$ which
uses the Bayes predictive density estimate from prior $\pi$ below the
threshold $\lambda_e$ and $\hat p_U$ above the threshold $\lambda_e$.
We bound the maximum predictive risk over $\me$:
%
%e43 #&#
\begin{equation}
\label{eq:decomp} \sup_{\pi\in\mathfrak{m}(\eta)} B(\pi, \hat p_T)
\leq(1- \eta) \rho(0, \hat p_T) + \eta\sup_{\theta}
\rho(\theta, \hat p_T).
\end{equation}
Next, as in \eqref{eq:th-decom}, we decompose the predictive risk of
$\hat p_T$ into contributions due to $X$ above and below the
threshold.
% \begin{eqnarray*}
% &\rho(\theta, \hat p_T)= \rho_A(\theta) + \rho_B(\theta), \mbox{
%where, }\\
% &\rho_A(\theta) = E_\theta[L(\theta, \hat p_U(\cdot|X)), |X| >
% \rho_B(\theta)= E_\theta[L(\theta, \hat p_\pi(\cdot|X)), |X| \leq
% \end{eqnarray*}
We calculate explicit expressions for $\rho_A$ and $\rho_B$.
The predictive loss of $\hat\rho_U$ (see Appendix~\ref
{A:gaussian_risk}) is given by
%
%e44 #&#
\begin{equation}
\label{eq:invloss} L\bigl(\theta, \hat p_U(\cdot|x)\bigr) =
a_{1r} + a_{2r}(\theta- x)^2
\end{equation}
with $a_{1r} = \hf[\log(1+r^{-1}) - (1+r)^{-1}]$ and
$a_{2r} = \hf(1+r)^{-1}$.
Hence, the above threshold term
%
%e45 #&#
\begin{equation}
\label{eq:rho-a} \rho_A(\theta) = a_{1r}
P_\theta\bigl(|X| > \lambda_e\bigr) + a_{2r}
E_\theta\bigl[(X-\theta)^2, |X|>\lambda_e\bigr].
\end{equation}
As $\rho_B(\theta)$ depends on the prior $\pi$ used below the
threshold, we restrict our attention to the specific choice of the
cluster prior. The risk functions of
the hard threshold density estimate $\hat p_{T,\pi_0}$ and that of
$\hat p_{T,\mathrm{LF}}$ can be easily derived from the calculations with the
cluster prior.

According to \eqref{eq:bayes_discrete} in the \hyperref[app]{Appendix}, the Bayes predictive
density for a discrete prior
$\pi=\sum_{k=-K}^K \pi_k \delta_{\mu_k}$ is given by
%
%e46 #&#
\begin{equation}
\label{eq:BayesP} \hat p_\pi(y|x) = \sum_{-K}^K
\phi(y|\mu_k,r) \pi_k \phi(x-\mu_k)/m(x),
\end{equation}
where $m(x) = \sum_k \pi_k \phi(x-\mu_k)$ denotes the marginal density
of $\pi$.
The K--L loss of $\hat p_\pi(\cdot|x)$ is given by
\[
L\bigl(\theta, \hat p_\pi(\cdot|x)\bigr) = \E_\theta\log
\frac{\phi(Y|\theta,r)}{\hat p_\pi(Y|x)}.
\]
A simple but informative upper bound for the K--L loss is obtained by
retaining only the $k$th term in \eqref{eq:BayesP}:
%
%e47 #&#
\begin{eqnarray}
\label{eq:second} L\bigl(\theta, \hat p_\pi(\cdot|x)\bigr) & \leq&
\E_\theta\log\frac{\phi(Y|\theta,r)}{\phi(Y|\mu_k,r)} - \log\frac{\pi_k \phi(x-\mu_k)}{\pi_0 \phi(x)} + \log
\frac{m(x)}{\pi_0 \phi(x)}
\nonumber
\\[-8pt]
\\[-8pt]
\nonumber
& =& \frac{1}{2r} (\theta- \mu_k)^2 +
\frac{1}{2} \bigl(\mu_k^2 - 2 x \mu_k
\bigr) - \log\frac{\pi_k}{\pi_0} + d(x),
\end{eqnarray}
where we have set
$d(x) = \log[m(x)/(\pi_0 \phi(x))]$.

We are now ready to analyze the bound \eqref{eq:decomp}.
We follow the steps recalled in the quadratic loss case [see
Section~S.4 of \citet{Supplementary}] and
evaluate the predictive risk at the origin and the maximum risk of the
threshold density estimate $\hat p_T$.
This organization helps to make clear the new features of the
predictive loss setting.

\textit{Risk at zero.} It is easy to show that $\rho(0, \hat p_T) =
O(\eta\lambda_f)$.
First, from \eqref{eq:rho-a}, we have
\[
\rho_A(0) = 2 a_{1r} \tilde\Phi(\lambda_e) +
a_{2r} q_A(0) = O(\eta\lambda_f),
\]
where $q_A(0)$ is defined in (S.4.2) and the above
calculation follows by using $\tilde\Phi(\lambda_e) \leq\lambda
_e^{-1} \phi(\lambda_e) =
O(\lambda_e^{-1} \eta)$ and the quadratic risk-at-zero bound
(S.4.4).

For the below-threshold term, we set $k = 0$ in \eqref{eq:second},
note that $\mu_0 = 0$ and apply Jensen's inequality to obtain
\[
\rho_B(0) = E_0 \bigl[ L\bigl(0,\hat
p_\pi(\cdot|X)\bigr), |X| \leq\lambda\bigr] \leq E_0
\bigl[d(X)\bigr] \leq\log E_0 \bigl[m(X)/\bigl(\pi_0
\phi(X)\bigr)\bigr].
\]
Since $E_0[m(X)/\phi(X)] = \int m(x) \,dx = 1$ and $\pi_0 = 1 - \eta$,
we obtain that
\[
\rho_B(0) \leq- \log(1- \eta) \leq\eta.
\]
%
% Also, we have:
% \begin{equation}\label{eq:exp(d)}
% \frac{m(x)}{\pi_0 \phi(x)}
% = 1 + \sum_{|k|=1}^K \frac{\pi_k}{\pi_0} \frac{\phi(x-\mu_k)}{\phi(x)}
% = 1 + \sum_{|k|=1}^K \frac{\pi_k}{\pi_0} \exp\bigg\{ x \mu_k - \frac{
% \end{equation}
% Now, we use one of the specific properties of the Cluster prior that $
% $$E_0 [m(X)/(\pi_0 \phi(X))] = 1+\sum_{|k|=1}^K \pi_k/\pi_0 = (1-
Consequently, $\rho_B(0) = O(\eta)$ and so $\rho(0,\hat p_{T,\mathrm{CL}}) =
O(\eta
\lambda_f)$. Note that the above calculations hold for any $\hat
p_{T,\pi}$ with $\pi$ being a discrete prior in $\me$.

\textit{Maximum risk.}
From decomposition \eqref{eq:decomp},
our goal is to show that
%
%e48 #&#
\begin{equation}
\label{eq:goal} \sup_{\theta} \rho(\theta, \hat p_{T,\mathrm{CL}})
= (2r)^{-1} \lambda_f^2 \bigl(1+o(1)\bigr).
\end{equation}
We first isolate the main term in the contributions from
$\rho_A(\theta)$ and $\rho_B(\theta)$.
From~\eqref{eq:rho-a}, clearly $\rho_A(\theta) \leq a_{1r} + a_{2r} =
O(1)$, which does not contribute.
We turn to
\[
\rho_B(\theta) = E_\theta\bigl[L\bigl(\theta, \hat
p_\pi(\cdot|X)\bigr), |X| \leq\lambda_e\bigr]
\]
and returning to \eqref{eq:second},
we begin by claiming
%let us also observe
that for $|x| \leq\lambda_e$ the final term $d(x) \leq
\log2$. Indeed,
%
%e49 #&#
\begin{equation}
\label{eq:exp(d)} \frac{m(x)}{\pi_0 \phi(x)} = 1 + \sum_{|k|=1}^K
\frac{\pi_k}{\pi_0} \frac{\phi(x-\mu
_k)}{\phi(x)} = 1 + \sum_{|k|=1}^K
\frac{\pi_k}{\pi_0} \exp \biggl\{ x \mu_k - \frac
{\mu_k^2}{2} \biggr
\}.
\end{equation}
For $|x| \leq\lambda_e$, we have
\[
x \mu_k - \mu_k^2/2 \leq
\lambda_e |\mu_k| - \mu_k^2/2 \leq
\lambda_e^2/2 = \log\eta^{-1}(1-\eta).
\]
%
%Using \eqref{eq:exp(d)} and
Since $\pi_0 = 1- \eta$, we arrive at
% $d(x) \leq\log2$.
%
%e50 #&#
\begin{equation}
\label{eq:d-bound} E_\theta\bigl[d(X), |X| \leq\lambda_e\bigr]
\leq\log2.
\end{equation}

The dependence of \eqref{eq:second} on $\theta$ may then be seen by
writing $x = \theta+ z$. The first two terms in \eqref{eq:second}
then take the form
\[
\frac{1}{2r} \bigl\{ \bigl[\theta-(1+r)\mu_k
\bigr]^2 - \bigl(r^2+r\bigr) \mu_k^2
\bigr\} - \mu_k z,
\]
while, after recalling that $\pi_k = \eta/(2K)$ and that
$\lambda_e^2 = 2 \log(1-\eta) \eta^{-1}$, the third term becomes
\[
\frac{1}{2} \lambda_e^2 + \log(2K) =
\frac{1}{2r} (1+r) \lambda_f^2 + \log(2K).
\]
We may therefore rewrite \eqref{eq:second} as
%
%e51 #&#
\begin{equation}
\label{eq:third} L\bigl(\theta, \hat p_\pi(\cdot|x)\bigr) \leq
\frac{1}{2r} \bigl[\lambda_f^2 + q_k(
\theta)\bigr] - \mu_k(x-\theta) + \log(2K) + d(x),
\end{equation}
where the $k${th} quadratic polynomial
\[
q_k(\theta) = \bigl[\theta- (1+r)\mu_k
\bigr]^2 - r^2 \mu_k^2 + r\bigl(
\lambda_f^2 - \mu_k^2\bigr).
\]
Denote the last three terms of \eqref{eq:third} by $J_k(x,\theta)$.
From \eqref{eq:support} and \eqref{eq:d-bound} we see that
\[
E_\theta[J_k, |X| \leq\lambda_e] \leq
\mu_k + \log(2K) + \log2 \leq\lambda_e + a + \log(4K) = o
\bigl(\lambda_f^2\bigr).
\]
Consequently, we obtain the key bound
%
%e52 #&#
\begin{equation}
\label{eq:main-bound} % E_\theta[L(\theta, \hat p_\pi(\cdot|X)), |X| \leq\lambda_e]
\rho_B(\theta) \leq\frac{1}{2r}
\Bigl[ \lambda_f^2 + \min_k
q_k(\theta) \Bigr] + o\bigl(\lambda_f^2
\bigr).
\end{equation}

Now we use the geometric structure of the support points $\mu_k$,
defined at \eqref{eq:support}.
We bound $\min_k q_k(\theta)$ above by considering the quadratic
polynomial $q_k(\theta)$ on
$I_k = [\mu_k, \mu_{k+1}]$ and
observe that these $2K$ intervals cover the range $(-\lambda_e-a,
-\lambda_f) \cup(\lambda_f,
\lambda_e + a)$ of interest.
See Figure~\ref{fig:second}.
Note that $q_k(\theta)$ achieves its
maximum on $I_k$ at both endpoints and that
%Further, from REF
%
\[
q_k(\mu_{k+1}) = q_k\bigl((1+2r)
\mu_k\bigr) = q_k(\mu_k) = r \bigl(
\lambda_f^2 - \mu_k^2\bigr).
\]
These maxima decrease with $k$ and so are bounded by
$q_1(\nu_\eta) = r( \lambda_f^2 - \nu_\eta^2)$.
Appealing now to bound \eqref{eq:diffbound}, we have
for $\lambda_f \leq|\theta| \leq\lambda_e + a$,
\[
\min_k q_k(\theta) \leq r\bigl(
\lambda_f^2 - \nu_\eta^2\bigr)
\leq2r \sqrt{v_w} a \lambda_f.
\]

Returning to \eqref{eq:main-bound}, we now see that the last two terms
are each $o(\lambda_f^2)$ and so the final bound \eqref{eq:goal} is proven.
This completes the proof of Lemma~\ref{lem:up-bound-sec4}.

% It is important to note on the role played by the geometric structure
%of the support points of the cluster prior which led to the asymptotic
%domination of the difference $\rho_B (\theta) - \lambda_f^2/2r$ by
%each of the quadratic polynomial $q_k(\theta)$. The maximum value of
%$q_k$ in the interval $[\mu_k,\mu_{k+1}]$ is given by $r (\lambda_f^2
%- \mu_k^2)$ which decreases as $|k|$ increases and ultimately plummets
%to $ r (\lambda_f^2 - \mu_{K}^2)$ which is a negative quantity iff $K
%negligible (see Figure~\ref{fig:second}).

%f2 #&#
\begin{figure}

\includegraphics{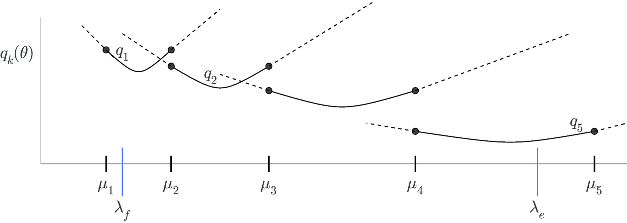}

\caption{Schematic diagram demonstrating the behavior of the quadratic
polynomials $q_k(\theta)$ in the interval $[\mu_1,\mu_{K+1}]$.
Here $K=4$.
The maximum of $\min_k q_k(\theta)$ for $\theta\in[\mu_1,\mu_{K+1}]$
is bounded by $q_1(\mu_1)$.}
\label{fig:second}
% \caption{Schematic diagram demonstrating the behavior of the
%quadratic polynomials $q_k(\theta)$ in the interval $[\mu_1,
%the role of $q_k(\theta)$ in upper bounding the risk of $\phat_{T,CL}$
%is based on $K=4$. }\label{fig:second}
\end{figure}

These calculations apply to threshold density estimates based on Bayes
estimates of discrete priors. In particular, for $\phat_{T,\mathrm{LF}}$ which
is based on the $3$-point prior $\pi_3[\eta,\nu_{\eta}]$, we have
$K=1$ and the bound \eqref{eq:three-pt-bound}. Thus, the difference
$\rho_B (\theta) - \lambda_f^2/2r$ in this case is negligible when
$|\theta| \leq\mu_2$.
% Calculations similar to \eqref{eq:third} and \eqref{eq:main-bound}
%will show that
% \begin{eqnarray}\label{eq:LF_risk}
% \rho(\theta, \hat p_{T,LF}) \approx\frac{1}{2r}
% \big\{ \lambda_f^2 - (\theta- \lambda_f)[(1+2r)\lambda_f - \theta]
% + o(\lambda_f^2) \mbox{ for } |\theta| \in[\lambda_f,\lambda_e]
% \end{eqnarray}
% and so, $\p_{T,LF}$ will be minimax sub-optimal when $\mu_1\sim
% 1+2r)\lambda_f > \lambda_e$ which happens iff $r < 0.4196$ (see
% Table~\ref{table:support.points}).

Similarly, the asymptotic risk function of the hard threshold plug-in
density estimate $\hat p_{T,\pi_0}$ (for which $K=0$ in our
calculations above) exceeds the minimax risk $\beta(\eta,r)$ for
$|\theta| \in[\lambda_f,\lambda_e]$ and so is minimax suboptimal for
any fixed $r$. Figure~\ref{fig:risk-simu} shows the numerical
evaluation of the risk functions for the different univariate
threshold density estimates.

%f3 #&#
\begin{figure}

\includegraphics{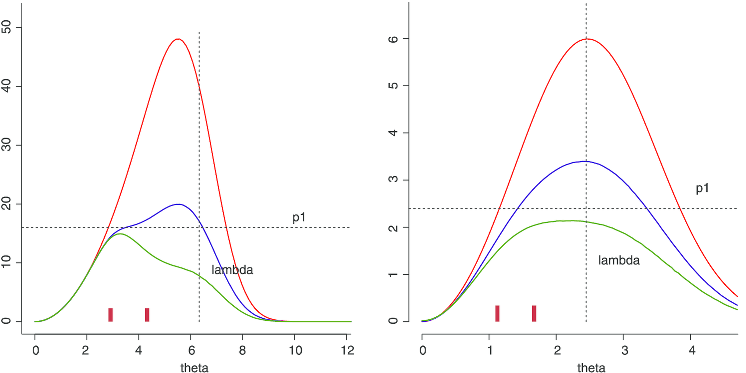}

\caption{Numerical evaluation of the asymptotic risk $\rho_B(\theta)$
for $r = 0.25 $ of univariate threshold density estimates:
hard threshold plug-in estimate $\hat p_{T,\pi_0}$ (red), $\hat
p_{T,\mathrm{LF}}$ (green) and the cluster prior-based minimax optimal estimate
$\hat p_{T,\mathrm{CL}}$ (blue). The brown boxes show the nonzero support point
of the cluster prior and the univarate asymptotic minimax risk $\beta
(\eta,r) = (2r)^{-1}\lambda_f^2$ and the threshold $\lambda_e$ are
respectively denoted by dotted horizontal and vertical lines. The plot
on left has $\eta=e^{-20}$ (very high sparsity), $\lambda_f=2.83$,
$\lambda_e=6.32$ and the right one has $\eta=0.05$ (moderate sparsity),
$\lambda_f=1.09$, $\lambda_e=2.45$.}\label{fig:risk-simu}
\end{figure}

Also, note that any threshold estimate $\hat p_{T}[\lambda]$ with
threshold size $\lambda$ less than $\lambda_e$ will be minimax
suboptimal, as its risk at the origin will not be negligible as
compared to $\beta(\eta,r)$. By \eqref{eq:th-decom} and
\eqref{eq:rho-a} we have
%
%e53 #&#
\begin{eqnarray}
\label{eq:min_threshold_size} \rho\bigl(0,\hat p_T[\lambda]\bigr) &\geq&2
a_{2r} E\bigl[Z^2 I\{Z > \lambda\} \bigr] = 2
a_{2r} \bigl\{\lambda\phi(\lambda) + 2 \tilde{\Phi}(\lambda)\bigr\}
\nonumber
\\[-8pt]
\\[-8pt]
\nonumber
&\geq& \lambda\phi(\lambda)/(1+r),
\end{eqnarray}
and so for any fixed $\varepsilon> 0$,
\[
\liminf_{1 \leq\lambda< \lambda_e(\eta) - \varepsilon} \frac{
\rho(0,\hat p_T[\lambda])}{\beta(\eta,r)} \to\infty.
\]
Thus, $\hat p_{T}[\lambda]$ is suboptimal unless $\lambda\geq
\lambda_e$.

%%% Local Variables:
%%% mode: latex
%%% TeX-master: "main"
%%% End:

%s5 #&#
\section{Theorem~\texorpdfstring{\protect\hyperref[th:mmx-risk]{1}}{1}: Multivariate minimax risk}\label{sec-multivariate}

Here we will use the univariate minimax results developed in the
previous sections to evaluate the asymptotic multivariate minimax risk
$R_n=R_{N}(\Theta_n[s_n])$ over the sparse parameter space
$\Theta_n[s_n]$.

%s5.1 #&#
\subsection{Lower bound proof: Theorem~\texorpdfstring{\protect\ref{th:least-fav}}{1B} and an extension}
\label{sec:theorem-1b:-lower}

We first prove a lower bound for the multivariate minimax risk under
only the assumption that $s_n/n \to0$---without requiring, as in
Theorem~\ref{th:least-fav}, that also $s_n \to\infty$.
This is done using an ``independent blocks'' sparse prior, along the
lines of \citet{Johnstone-book}, Chapter~8.6, that we will show to be asymptotically least
favorable.
This result establishes the lower bound half of Theorem~\ref{th:mmx-risk}.
At the end of the subsection, we prove Theorem~\ref{th:least-fav} using the
simpler i.i.d. prior.

% We begin by defining
% under the assumption $s_n/n \to0$ of Theorem~\ref{th:mmx-risk}.

Let $\pi_S(\tau;m)$ denote a single spike prior of scale $\tau$ on
$\mathbb{R}^m$: choose an index $I \in\{1, \ldots, m \}$ at random
and set $\theta= \tau e_I$, where $e_I$ is a unit length vector in
the $i$th coordinate direction.
We will use a scale $\tau_m = \lambda_m - \log\lambda_m$ which is
somewhat smaller than $\lambda_m = \sqrt{2 \log m}$.

The independent blocks prior $\pi^\mathrm{ IB}$ on $\Theta[s_n]$ is built
by dividing $\{1, \ldots, n\}$ into $s_n$ contiguous blocks
$B_j, j = 1, \ldots, m$ each of length $m = m_n = [n/s_n]$.
Draw components $\theta_i$ in each block $B_j$ according to an
independent copy of $\pi_S(\nu_m; m)$ where the scale $\nu_m = \sqrt
v_w \tau_m$ is matched to the prediction setting.
Finally, set $\theta_i = 0$ for the remaining $n - m_n s_n$
components.
Thus, $\pi^\mathrm{ IB}$ is supported on $\Theta[s_n]$ since any draw
$\theta$ from $\pi^\mathrm{ IB}$ has exactly $s_n$ nonzero components.

The lower bound half of Theorem~\ref{th:mmx-risk} follows from the
following result, the analog of Theorem~\ref{th:least-fav} for the
independent blocks prior.

%th6 #&#
\begin{theorem}
\label{th:mult-lower-bound}
Fix $r \in(0,\infty)$. If $s_n/n \to0$, then
\[
R_N\bigl(\Theta_n[s_n]\bigr) \geq B\bigl(
\pi_n^\mathrm{ IB}\bigr) \geq(1+r)^{-1} s_n
\log(n/s_n).
\]
\end{theorem}

\begin{pf}
Bounding maximum risk by Bayes risk and using the product structure
shows that
%
%e54 #&#
\begin{equation}
\label{eq:first} R_{n} = R_{N}\bigl(\Theta_n[s_n]
\bigr) \geq B\bigl(\pi_n^\mathrm{ IB}\bigr) = s_n B
\bigl(\pi_S(\nu_m; m)\bigr).
\end{equation}
Next, using $B_Q^v$ to denote the Bayes risk for noise level $v$, the
multivariate form of the connecting equation and scale invariance
enable us to write
\[
B\bigl(\pi_S(\nu_m; m)\bigr) = \frac{1}{2} \int
_{v_w}^1 B_Q^v\bigl(
\pi_S(\nu_m;m)\bigr) \frac{dv}{v^2} =
\frac{1}{2} \int_{v_w}^1 B_Q
\biggl(\pi_S \biggl(\frac{\nu_m}{\sqrt
v};m \biggr) \biggr)
\frac{dv}{v}.
\]
The next lemma, proved in Section~S.5 of \citet
{Supplementary}, provides a uniform lower bound for the quadratic loss
Bayes risk of a single spike prior.
It is a multivariate analog of Lemma~\ref{lem:qr}.

%pr7 #&#
\begin{proposition}
\label{prop:spikebound}
Suppose that $y \sim N_n(0,I)$.
Set $\lambda_n = \sqrt{2 \log n}$
and $\tau_n = \lambda_n - \log\lambda_n$.
% Let $\pi_S(\tau;n)$ denote the single spike prior which puts $\theta=
% \tau e_I$ for $I$ uniformly chosen on $\{1, \ldots, n \}$.
% Let
Then there exists $\varepsilon_n \to0$ such that uniformly in $\tau
\in
[0, \tau_n]$,
\[
B_q\bigl( \pi_S(\tau;n)\bigr) \geq\tau^2
(1 - \varepsilon_n).
\]
\end{proposition}

Noting that $v \in
[v_w,1]$ implies that $\nu_m/\sqrt v \leq\nu_m/\sqrt{v_w} = \tau_m$,
and then applying the proposition,
\[
B\bigl(\pi_S(\nu_m; m)\bigr) \geq\frac{(1-\varepsilon_m)}{2}
\int_{v_w}^1 \frac{\nu_m^2}{v^2} \,dv = (1 -
\varepsilon_m) \frac{\nu_m^2}{2r}.
\]
Combining this with \eqref{eq:first} and the definition of
$\nu_m$, we obtain
%
%e55 #&#
\begin{equation}
\label{eq:m-lower} R_{n} \geq(1-\varepsilon_m)
s_n v_w \tau_m^2 / (2r)
\sim(1+r)^{-1} s_n \log(n/s_n). %\qedhere
\end{equation}
\upqed\end{pf}

\begin{pf*}{Proof of Theorem~\ref{th:least-fav}}
Note that because of the product
structure of the problem and the prior
$\pi_n^{\mathrm{IID}}$ we have
\[
B\bigl(\pi_n^{\mathrm{IID}}\bigr)=\sum_{i=1}^n
\beta(\eta_n,r)=n \beta(\eta_n,r),
\]
which is asymptotically equal to $R_N(\Theta[s_n])$,
using the univariate Theorem~\ref{th:univ-mmx-result}
[cf.~\eqref{eq:beta-rate}] and
%
%e56 #&#
\begin{equation}
\label{eq:lam-f-rate} (2r)^{-1} \lambda_f^2 =
(2r)^{-1} v_w \lambda_e^2 \sim
(1+r)^{-1} \log\eta_n^{-1}\qquad \mbox{as } n \to
\infty.
\end{equation}
Also, as $s_n \to\infty, \pi_n^\mathrm{ IID}(\Theta[s_n]) \to1$ by
application of Chebyshev's inequality and, hence, $\pi_n^{\mathrm{IID}}$ is an
asymptotically least favorable prior under the conditions of
Theorem~\ref{th:least-fav}.
\end{pf*}
%

%s5.2 #&#
\subsection{Upper bound proof: Theorem~\texorpdfstring{\protect\ref{th:asy-mmx-rule}}{1C}}
\label{sec:theorem-1c:-upper}

First, an upper bound on\break $R_{N}(\Theta_n[s_n])$
is derived based on the maximum risk of the multivariate product
threshold density estimate $\hat p_{T,\mathrm{CL}}$ defined in Theorem~\ref{th:asy-mmx-rule}.
Using the product structure of the threshold estimate as well as that
of the unknown future density
\[
\hat p_{T,\mathrm{CL}}( y| x) = \prod_{i=1}^n
\hat p_{T,\mathrm{CL}}(y_i|x_i) \quad\mbox{and}\quad p(y|
\theta, r)=\prod_{i=1}^n
p(y_i|\theta_i,r),
\]
the risk of our multivariate threshold estimate simplifies as an
agglomerative coordinate wise risk of the respective univariate density
estimates
\[
\rho(\theta, \hat p_{T,\mathrm{CL}}) =
 E_{\theta} \log\frac{p(y| \theta, r)}{\hat p_{T,\mathrm{CL}} (y| x)} =
\sum_{i=1}^n \rho( \theta_i,
\hat p_{T}).
\]
Now, maximizing over $\theta\in\Theta_n[s_n]$, we have
\[
R_{n} \leq\sup_{\Theta_n[s_n]} \rho( \theta, \hat
p_{T,\mathrm{CL}}) \leq(n-s_n) \rho( 0, \hat p_{T,\mathrm{CL}}) +
s_n \sup_\theta\rho(\theta, \hat p_{T,\mathrm{CL}}).
\]

From the univariate study, we know that $\rho(0, \hat p_{T,\mathrm{CL}}) =
O(\eta_n \lambda_f)$, which makes $(n-s_n) \rho( 0, \hat p_{T,\mathrm{CL}})
=O(s_n \lambda_f)$ negligible relative to
\[
s_n \sup_\theta\rho( \theta, \hat
p_{T,\mathrm{CL}}) = (2r)^{-1} s_n \lambda_f^2
\bigl(1 + o(1)\bigr),
\]
where we used \eqref{eq:goal}.
Thus, taking account also of \eqref{eq:lam-f-rate}, we have the
desired upper bound on the minimax risk
%
%e57 #&#
\begin{equation}
\label{eq:m-upper} R_{n} \leq(2r)^{-1} s_n
\lambda_f^2 \bigl(1 + o(1)\bigr) \sim(1+r)^{-1}
s_n \log(n/s_n).
\end{equation}

\textit{Completion of Proof of Theorems~\ref{th:mmx-risk}, \ref{th:least-fav}
and \ref{th:asy-mmx-rule}}:
As the lower bound \eqref{eq:m-lower}
and upper bound \eqref{eq:m-upper} on $R_n$ match asymptotically,
the first order asymptotic minimax risk of Theorem~\ref{th:mmx-risk}
is achieved, and the proof of all parts is done.
% which along with the above constructive proof of the
% bounds also proves
% Theorem~\ref{th:asy-mmx-rule}.
%Theorems \ref{th:least-fav-finite-s} and \ref{th:asy-mmx-rule}.
% For the proof of Theorem~\ref{th:least-fav}, note that because of the product
% structure of the problem and the prior
% $\pi_n^{IID}$ we have,
% $$B(\pi_n^{IID})=\sum_{i=1}^n \beta(\eta_n,r)=n \beta(\eta_n,r)$$
% which is asymptotically equal to $R_N(\Theta[s_n])$. Also, as $s_n
%Chebyshev's inequality and hence $\pi_n^{IID}$ is an asymptotically
%least favorable prior under the conditions of Theorem~\ref{th:least-fav}.

%s5.3 #&#
\subsection{Proof of Proposition~\texorpdfstring{\protect\ref{prop:further-results-1}}{1}}
\label{sec:fill}

Estimates in $\mathcal{L}$ and $\mathcal{G}$ are products of the form~\eqref{eq:prod-rule} and so
$R_{\mathcal{L},n}=R_{\mathcal{L}}(\Theta_n[s_n])$ can be studied
using the associated univariate problem and decomposition
\eqref{eq:multi-sup}.
It is shown in Appendix~\ref{A:gaussian_risk} that
%
%e58 #&#
\begin{equation}
\label{eq:lin-risk} \rho(\theta, \phat_{L,\alpha}) = \frac{1}{2} \log
\biggl(1 + \frac{\alpha}{r} \biggr) + \frac{(1-\alpha)^2}{2(r+\alpha)} \biggl[
\theta^2 - \frac{\alpha}{1-\alpha} \biggr].
\end{equation}
Thus, $\sup_\theta\rho(\theta, \phat_{L,\alpha})$ is infinite unless
$\alpha= 1$,
%---the uniform prior estimate $\hat p_U$---
that is, the uniform prior estimate $\hat p_U$,
in which case
$\rho(\theta, \phat_U) \equiv\frac{1}{2} \log(1+r^{-1})$. Thus,
\[
R_{\mathcal{L},n}= \frac{n}{2} \log\bigl(1+r^{-1}\bigr) \gg
\frac{s_n}{1+r} \log \biggl(\frac{n}{s_n} \biggr) \sim R_n.
\]
In particular, $R_{\mathcal{L},n}/R_n \to\infty$ when $s_n/n \to0$.

We turn to the Gaussian class $\mathcal{G}$.
Since $\mathcal{E} \subset\mathcal{G}$, clearly
$R_{\mathcal{G},n} \leq R_{\mathcal{E},n}=(2r)^{-1} n \eta_n
\lambda_e^2$.
We give here a heuristic argument for the reverse inequality, which
gives the idea for the rigorous proof given in Section~S.3 of the supplementary material [\citet{Supplementary}].
From the decomposition \eqref{eq:multi-sup}, any near-optimal
estimator in $\mathcal{G}$ must have univariate risk at $0$ bounded as
follows:
%
%e59 #&#
\begin{equation}
\label{eq:uni_risk_bound_prop_1} \rho( 0, \hat p_1) \leq r^{-1}
\eta_n \lambda_e^2.
\end{equation}
Now from \eqref{eq:gaussian_risk} we know that the risk at the origin
for the univariate Gaussian density estimate $p[\that,\dhat]$ is
\[
\rho\bigl(0,p[\that,\dhat]\bigr)=2^{-1}E_0\bigl\{\log
\bigl(r^{-1}\dhat\bigr)+\dhat ^{-1}\bigl(r+\that
^2\bigr)-1\bigr\}, %
\]
which for any fixed choice of $\that$ achieves its minimum at
$d_{\opt}[\that]=r+\that^2$.
Thus, for such an optimal choice of $\dhat$,
\[
\rho\bigl(0,p\bigl[\that, d_{\opt}(\that)\bigr]\bigr) =
E_0 \log\bigl(1+r^{-1}\that^2\bigr),
\]
and for this to satisfy \eqref{eq:uni_risk_bound_prop_1}, we must have
$\that(x) \approx0$ for $|x| \leq\lambda_e(1+o(1))$.
Thus, $\phat$ would approximately need to have the threshold structure
\eqref{eq:p-hat-T},~\eqref{eq:zero-prior} for $|x| \leq\lambda_e$ and
so the bound \eqref{eq:quadratic} shows that
\[
\rho(\theta,\phat_1) \geq\frac{\theta^2}{2r} P_{\theta}\bigl(|X|
\leq\lambda_e\bigr) \sim\frac{\lambda_e^2}{2r}.
\]
Returning to decomposition \eqref{eq:multi-sup}, we can now see that
$R_{\mathcal{G},n} \gtrsim(2r)^{-1} s_n \lambda_e^2 \sim
R_{\mathcal{E},n}$, which completes the heuristic argument.

\section{Discussion}

\textit{Avoiding thresholding.}
The asymptotic minimax rules $\hat p_T$ described in Theorems
\ref{th:asy-mmx-rule} and \ref{th:univ-mmx-result}\textsc{c} are based on
thresholding.
It would be desirable to construct a prior $\pi$ for which the Bayes
predictive density $\hat p_\pi$ in (\ref{eq:bpd}) is itself
asymptotically minimax, without any use of the discontinuous thresholding
operation.

Consider, then, a symmetric univariate prior $\pi_\infty[\eta,r]$
whose support consists of the origin and infinite number of
equidistant clusters each containing $2K$ points
%(defined before in equation~[\ref{cluster.prior.defn}] )
in the same spatial alignment as for $\pi_{\mathrm{CL}}[\eta,r]$:
\[
\pi_\infty[\eta,r] = (1-\eta) \delta_0 + \frac{1-\eta}{2}
\sum_{j=0}^\infty\eta^{j+1} \sum
_{k=1}^K q_k (
\delta_{\mu_{jk}} + \delta_{-\mu_{jk}}),
\]
where
$\mu_{jk} = j \lambda_e + \mu_k$ and for $k = 2, \ldots, K$ and
$\gamma
= \log\eta^{-1}$, we have
$ q_k = \gamma^{-k}$ and $q_1 = 1 - \sum_2^K q_k$.
%
%q_k = \gamma^{-k},\qquad q_1 = 1 - \sum
%_2^K q_k.

Based on $\pi_\infty[\eta_n,r]$, one can construct a multivariate prior
$\pi_{n,\infty}^\mathrm{ IID}$ using (\ref{eq:iid-prior}), which
heuristic arguments indicate will not
only be least favorable but also yield a minimax optimal density
estimate. A detailed proof is forthcoming.

\textit{Approximate sparsity and other extensions}.
Starting from
% Following the lines of
\citet{Johnstone-book}, Chapters~8 and 13, the $\ell_0$
sparsity results presented here can be extended to obtain minimax
optimal predictive density estimates over\vadjust{\goodbreak} weak and strong $\ell_p$ sparse
parameter spaces.
An interesting topic for future work will be whether, as in
point estimation
[\citet{Donoho94b}], the phenomena seen here can be generalized to a
family of loss functions. Simple analogues of the connecting equations
[\citet{Brown08}, Theorem~1] between the predictive and quadratic PE
regimes do not exist in those cases, though some of the decision
theoretic parallels can still be proved particularly for the $\ell_2$
loss [\citet{Gatsonis84}].

%
%Under the conditions of Theorem~\ref{th:mmx-risk} for any fixed $r \in
%(0,\infty]$, the proper prior distribution
%$\pi_{n,\infty}^\mathrm{ IID}$
%%$\pi[n,s,r,\Inf]$
%is asymptotically least favorable and its corresponding Bayes
%predictive
%density is asymptotically minimax optimal.

%%% Local Variables:
%%% mode: latex
%%% TeX-master: "main"
%%% End:

\begin{appendix}\label{app}
\section*{Appendix}

% \subsection{K-L loss for uniform prior}\label{A:uniform_prior}
% The estimator $\hat p_U$ is given by
% the $N(x, 1+r)$ distribution, and so
% \begin{eqnarray*}
% E_\theta\log\hat p_U(Y|x)
% & = - \frac{1}{2} \log[2 \pi(1+r)] - \frac{1}{2(1+r)}[(\theta-x)^2
% + r]. \\
% \intertext{Since $E_\theta\log\phi(Y|\theta,r)
% = - \frac{1}{2} \log(2 \pi r) - \frac{1}{2}$, we obtain}
% L(\theta, \hat p_U(\cdot|x))
% & = \frac{1}{2} \log(1+r^{-1}) + \frac{1}{2(1+r)} (\theta- x)^2 +
% \frac{1}{2} - \frac{r}{2(1+r)},
% \end{eqnarray*}
% from which \eqref{eq:invloss} is immediate.
%s6.1 #&#
\subsection{Bayes density estimate for discrete priors} The posterior
distribution for the discrete prior $\pi=\sum_{k=-K}^K \pi_k \delta
_{\mu
_k}$ is given by
\[
\pi(\mu_k|x)= \bigl\{ m(x)\bigr\}^{-1} \phi(x|
\mu_k,1) \pi_k\qquad \mbox{where }m(x)=\sum
_{k} \pi_k \phi(x|\mu_k,1).
\]
So, for the Bayes predictive density based on the prior $\pi$,
%
%e60 #&#
\begin{equation}
\label{eq:bayes_discrete} \hat p_{\pi}(y|x) =\sum_{k=-K}^K
\phi(y|\mu_k,r) \pi(\mu_k|x) =\sum
_{k=-K}^K \phi(y|\mu_k,r)
\frac{\phi(x|\mu_k,1) \pi_k}{m(x)}.
\end{equation}

%s6.2 #&#
\subsection{K--L risk for gaussian and linear density
estimates}\label{A:gaussian_risk}

The predictive risk of the univariate Gaussian density estimate
$p [ \that,\dhat]=N(\that,\dhat)$ is given by
\[
\rho\bigl(\theta,p [ \that,\dhat]\bigr)=E_{\theta}\bigl\{ \log\phi(Y|
\theta,r) \bigr\} - E_{\theta}\bigl\{ \log\phi\bigl(Y|\that(X), \dhat(X)
\bigr)\bigr\}, %
\]
where the expectation is over $X \sim N(\theta,1)$ and $Y \sim
N(\theta,r)$. Noting that
\[
E_{\theta}\bigl\{\log\phi(Y|\that, \dhat)|X=x\bigr\}= - \tfrac{1}{2}
\log\bigl(2 \pi\dhat(x)\bigr) - \bigl(2 \dhat(x)\bigr)^{-1}\bigl\{r+
\bigl(\that(x)-\theta\bigr)^2 \bigr\}
\]
and $E_\theta\log\phi(Y|\theta,r)
= - \frac{1}{2} \log(2 \pi r) - \frac{1}{2}$, we obtain
%
%e61 #&#
\begin{equation}
\label{eq:loss-form} L\bigl(\theta, \phat(\cdot|x)\bigr) % = \frac{1}{2} \log\biggl(\frac{\hat d}{r} \biggr)
=
\frac{1}{2} \log\bigl(r^{-1} \hat d\bigr) + \frac{r+(\that(x)-\theta)^2}{2 \hat d} -
\frac{1}{2},
\end{equation}
and the following expression for the K--L risk of members in $\mathcal{G}$:
%
%e62 #&#
\begin{equation}
\label{eq:gaussian_risk} \rho\bigl(\theta,p [ \that,\dhat]\bigr)=\frac{1}{2} \biggl[
E_{\theta} \log \bigl(r^{-1}\dhat\bigr) + E_{\theta} \biggl
\{ \frac{ r+(\that-\theta
)^2}{\dhat} \biggr\} - 1 \biggr].
\end{equation}

Consider now ``linear'' estimators. Starting with the conjugate prior
$\theta\sim N(0, \alpha/(1-\alpha))$ for $0 \leq\alpha\leq1$,
standard calculations show that the posterior density
$\pi(\theta|x)$ is $N(\alpha x, \alpha)$ and the predictive density
$\phat_{L,\alpha}$, being the convolution of Gaussians, compare
\eqref{eq:bpd}, is seen to be $N(\alpha x, r+\alpha)$.
Now, using
$\dhat=r+\alpha$ and $\that=\alpha X$ in \eqref{eq:gaussian_risk},
we get
\[
\rho(\theta,\phat_{L,\alpha})=\tfrac{1}{2} \bigl[ \log
\bigl(1+r^{-1}\alpha\bigr) +(r+\alpha)^{-1} \bigl\{
r+E_{\theta}(\alpha X-\theta)^2\bigr\} - 1 \bigr].
\]
The linear risk formula~\eqref{eq:lin-risk} now follows from the
quadratic risk of
$\alpha X$. Next, we present some details about the risk of the
particular linear estimate $\phat_U$.

\textit{Proof of} \eqref{eq:invloss}.
The estimator $\hat p_U = \phat_{L,1}$ is given by
the $N(x, 1+r)$ distribution, and so from \eqref{eq:loss-form}
\[
L\bigl(\theta, \hat p_U(\cdot|x)\bigr) = \frac{1}{2} \log
\bigl(1+r^{-1}\bigr) + \frac{r + (\theta- x)^2}{2(1+r)} - \frac{1}{2},
\]
from which \eqref{eq:invloss} is immediate.
\end{appendix}
%%% Local Variables:
%%% mode: latex
%%% TeX-master: "main"
%%% End:

\section*{Acknowledgments}
We thank the Associate Editor and three referees for constructive
suggestions to shorten and improve the paper.


\begin{supplement}%[id=suppA]
\stitle{Supplementary material to ``Exact minimax estimation of the
predictive density in sparse Gaussian models''}
\slink[doi]{10.1214/14-AOS1251SUPP} %[doi,text={...}] - jei reikia
%suskaldyti doi
\sdatatype{.pdf}
\sfilename{aos1251\_supp.pdf}
\sdescription{The supplement \citet{Supplementary} contains
a brief description of the relevance of the predictive
density estimation problem in related application
areas along with the proof for the suboptimality
of the univariate threshold density estimate
$\phat_{T,\mathrm{LF}}$ (in Section~S.2) and the
details of the proof of Proposition~\ref{prop:further-results-1}
(in Section~S.3). The arguments for
the maximum quadratic risk of hard threshold point
estimates are reviewed in Section~S.4
and the proof of Proposition~\ref{prop:spikebound} is
presented in Section~S.5.
Links to R-codes used in producing Table~\ref{table:support.points}
and Figure~\ref{fig:risk-simu} are also provided.}
\end{supplement}

% \sname{}\label{supp}
% \stitle{}
% \slink[url]{url}
% \sdescription{}
% \end{supplement}

% imsref loaded by akundreckaite, 2015-02-26 16:09:51
%


%suskaldyti doi

\printaddresses
\end{document}